\documentclass{amsart}
\usepackage{amsmath}
\usepackage{amssymb}
\usepackage{amsrefs}
\usepackage[all]{xy}
\usepackage{mathrsfs}
\numberwithin{equation}{section}

\newtheorem{thm}{Theorem}[section]

\newtheorem{lemma}[thm]{Lemma}
\newtheorem{prop}[thm]{Proposition}

\theoremstyle{definition}
\newtheorem{definition}[thm]{Definition}
\newtheorem*{notation}{Notation}

\theoremstyle{remark}
\newtheorem{remark}[thm]{Remark}
\newtheorem{example}[thm]{Example}
\def\mathcs{C^{*}}
\newcommand{\cs}{\ensuremath{\mathcs}}

\DeclareMathSymbol{\rtimes}{\mathbin}{AMSb}{"6F}

\def\lk{\langle}
\def\rk{\rangle}

\def\R{\mathbb{R}}
\def\C{\mathbb{C}}
\def\T{\mathbb{T}}
\def\Z{\mathbb{Z}}
\def\K{\mathcal{K}}

\DeclareMathOperator{\Otwo}{O(2)}
\DeclareMathOperator{\Ad}{Ad}
\DeclareMathOperator{\Ind}{Ind}

\DeclareMathOperator{\Prim}{Prim}

\DeclareMathOperator{\Aut}{Aut}

\DeclareMathOperator{\id}{id}
\def\set#1{\{\,#1\,\}}
\newcommand\sset[1]{\{#1\}}
\let\tensor=\otimes
\def\restr#1{|_{{#1}}}
\makeatletter
\def\labelenumi{\textnormal{(\@alph\c@enumi)}}
\def\theenumi{\@alph \c@enumi}
\def\labelenumii{\textnormal{(\@roman\c@enumii)}}
\def\theenumii{\@roman \c@enumii}
\newcount\charno
\def\alphapart#1{\charno=96
\advance\charno by#1\char\charno}

\makeatother

\def\<{\langle}
\def\>{\rangle}
\let\ipscriptstyle=\scriptscriptstyle
\def\lipsqueeze{{\mskip -3.0mu}}
\def\ripsqueeze{{\mskip -3.0mu}}
\def\ipcomma{\nobreak\mathrel{,}\nobreak}
\newbox\ipstrutbox
\setbox\ipstrutbox=\hbox{\vrule height8.5pt
width 0pt}
\def\ipstrut{\copy\ipstrutbox}
\def\lip#1<#2,#3>{\mathopen{\relax_{\ipstrut\ipscriptstyle{
#1}}\lipsqueeze
\langle} #2\ipcomma #3 \rangle}
\def\blip#1<#2,#3>{\mathopen{\relax_{\ipstrut
\ipscriptstyle{ #1}}\lipsqueeze\bigl\langle} #2\ipcomma #3 \bigr\rangle}
\def\rip#1<#2,#3>{\langle #2\ipcomma #3
\rangle_{\ripsqueeze\ipstrut\ipscriptstyle{#1}}}
\def\brip#1<#2,#3>{\bigl\langle #2\ipcomma #3
\bigr\rangle_{\ripsqueeze\ipstrut\ipscriptstyle{#1}}}
\def\angsqueeze{\mskip -6mu}
\def\smangsqueeze{\mskip -3.7mu}
\def\trip#1<#2,#3>{\langle\smangsqueeze\langle #2\ipcomma #3
\rangle\smangsqueeze\rangle_{\ripsqueeze\ipstrut\ipscriptstyle{#1}}}
\def\btrip#1<#2,#3>{\bigl\langle\angsqueeze\bigl\langle #2\ipcomma
#3
\bigr\rangle
\angsqueeze\bigr\rangle_{\ripsqueeze\ipstrut\ipscriptstyle{#1}}}
\def\tlip#1<#2,#3>{\mathopen{\relax_{\ipstrut\ipscriptstyle{
#1}}\lipsqueeze \langle\smangsqueeze\langle} #2\ipcomma #3
\rangle\smangsqueeze\rangle}
\def\btlip#1<#2,#3>{\mathopen{\relax_{\ipstrut\ipscriptstyle{
#1}}\lipsqueeze
\bigl\langle\angsqueeze\bigl\langle} #2\ipcomma #3
\bigr\rangle\angsqueeze\bigr\rangle}

\def\ip(#1|#2){(#1\mid #2)}
\def\bip(#1|#2){\bigl(#1 \mid #2\bigr)}
\def\Bip(#1|#2){\Bigl( #1 \bigm| #2 \Bigr)}
\def\om{\omega}
\expandafter\ifx\csname BibSpec\endcsname\relax\else
\BibSpec{collection.article}{%
    +{}  {\PrintAuthors}                {author}
    +{,} { \textit}                     {title}
    +{.} { }                            {part}
    +{:} { \textit}                     {subtitle}
    +{,} { \PrintContributions}         {contribution}
    +{,} { \PrintConference}            {conference}
    +{}  {\PrintBook}                   {book}
    +{,} { }                            {booktitle}
    +{,} { }                            {series}
    +{,} { \voltext}                    {volume}
    +{,} { }                            {publisher}
    +{,} { }                            {organization}
    +{,} { }                            {address}
    +{,} { \PrintDateB}                 {date}
    +{,} { pp.~}                        {pages}
    +{,} { }                            {status}
    +{,} { \PrintDOI}                   {doi}
    +{,} { available at \eprint}        {eprint}
    +{}  { \parenthesize}               {language}
    +{}  { \PrintTranslation}           {translation}
    +{;} { \PrintReprint}               {reprint}
    +{.} { }                            {note}
    +{.} {}                             {transition}
}
\BibSpec{article}{%
    +{}  {\PrintAuthors}                {author}
    +{,} { \textit}                     {title}
    +{.} { }                            {part}
    +{:} { \textit}                     {subtitle}
    +{,} { \PrintContributions}         {contribution}
    +{.} { \PrintPartials}              {partial}
    +{,} { }                            {journal}
    +{}  { \textbf}                     {volume}
    +{}  { \PrintDatePV}                {date}
    +{,} { \eprintpages}                {pages}
    +{,} { }                            {status}
    +{,} { \PrintDOI}                   {doi}
    +{,} { available at \eprint}        {eprint}
    +{}  { \parenthesize}               {language}
    +{}  { \PrintTranslation}           {translation}
    +{;} { \PrintReprint}               {reprint}
    +{.} { }                            {note}
    +{.} {}                             {transition}
}
\BibSpec{book}{%
    +{}  {\PrintPrimary}                {transition}
    +{,} { \textit}                     {title}
    +{.} { }                            {part}
    +{:} { \textit}                     {subtitle}
    +{,} { \PrintEdition}               {edition}
    +{}  { \PrintEditorsB}              {editor}
    +{,} { \PrintTranslatorsC}          {translator}
    +{,} { \PrintContributions}         {contribution}
    +{,} { }                            {series}
    +{,} { \voltext}                    {volume}
    +{,} { }                            {publisher}
    +{,} { }                            {organization}
    +{,} { }                            {address}
    +{,} { pp.~}                        {pages}
    +{,} { \PrintDateB}                 {date}
    +{,} { }                            {status}
    +{}  { \parenthesize}               {language}
    +{}  { \PrintTranslation}           {translation}
    +{;} { \PrintReprint}               {reprint}
    +{.} { }                            {note}
    +{.} {}                             {transition}
}
\fi
\newcommand{\xtg}{\ensuremath{X\rtimes G}}
\newcommand\usc{upper-semicontinuous}
\def\sa_#1(#2,#3){\Gamma_{#1}(#2,#3)}
\newcommand\bundlefont{\mathscr}
\newcommand\B{\bundlefont{B}}
\newcommand\A{\bundlefont{A}}
\newcommand\gux{G\backslash X}
\newcommand\lt{\operatorname{lt}}
\newcommand\kltg{\K\bigl(L^{2}(G)\bigr)}
\renewcommand\H{\mathcal{H}}
\newcommand\spec[1]{(#1)^{\wedge}}
\newcommand\ind{\operatorname{ind}}
\newcommand\stab{\operatorname{Stab}}
\newcommand\stabxbh{\ensuremath{\stab\spec{X_{\beta}}}}
\newcommand\stabybh{\ensuremath{\stab\spec{Y_{\beta}}}}
\newcommand\ltgvs{\mathcal{W}_{V,\sigma}}
\newcommand\ltgvt{\mathcal{W}_{V,\tau}}
\newcommand\ltgvxs{\mathcal{W}_{V_{x},\sigma}}
\newcommand\fix{\operatorname{fix}}
\newcommand\N{\mathbb{N}}
\newcommand\Rep{\operatorname{Rep}}
\newcommand\ab{\operatorname{ab}}
\newcommand\djunion{\amalg}

\allowdisplaybreaks

\newcount\hours
\newcount\minutes       
\def\timeofday{
\hours=\time
\minutes=\hours
\divide\hours by60
\multiply\hours by60
\advance\minutes by-\hours
\divide\hours by60
\ifnum\hours>9\else0\fi\the\hours:\ifnum\minutes>9\else
0\fi\the\minutes}
\def\predate{\date{\color{red}\bfseries \the\day\ \ifcase\month\or
  January\or February\or March\or April\or May\or June\or July\or
        August\or September\or October\or November\or
           December\fi\ \the\year\ --- \timeofday\ --- v3}}

\begin{document}

       \title[Crossed products by strictly proper actions] {Structure
         of crossed products by strictly proper actions on
         continuous-trace algebras}

       \author[Echterhoff]{Siegfried Echterhoff}
       \address{Westf\"alische Wilhelms-Universit\"at M\"unster,
         Mathematisches Institut, Einsteinstr. 62 D-48149 M\"unster,
         Germany} \email{echters@uni-muenster.de}

       \author{Dana P. Williams} \address{Department of Mathematics \\
         Dartmouth College \\ Hanover, NH 03755-3551}

       \email{dana.williams@Dartmouth.edu}


       \date{20 August 2012}

       \thanks{The research for this paper was partially supported by
         the German Research Foundation (SFB 478 and SFB 878) and the
         EU-Network Quantum Spaces Noncommutative Geometry (Contract
         No. HPRN-CT-2002-00280) as well as the Edward Shapiro Fund at
         Dartmouth College.}

\begin{abstract}
  We examine the ideal structure of crossed products
  $B\rtimes_{\beta}G$ where $B$ is a continuous-trace \cs-algebra and
  the induced action of $G$ on the spectrum of $B$ is proper.  In
  particular, we are able to obtain a concrete description of the
  topology on the spectrum of the crossed product in the cases where
  either $G$ is discrete or $G$ is a Lie group acting smoothly on the
  spectrum of~$B$.
\end{abstract}

\maketitle


\section*{Introduction} A first step in understanding the structure of
a crossed product $B\rtimes_{\beta}G$ is to describe its spectrum
together with its Jacobson topology.  Unfortunately, such a task is
hopelessly out of reach in any sort of general setting.  If we at
least assume that the action of $G$ on the spectrum of $B$ is
well-behaved, then the Mackey-Rieffel-Green machine allows us to
describe $\spec{B\rtimes_{\beta}G}$ as a \emph{set} in terms of
projective representations of the stability groups for the action of
$G$ on the spectrum of $B$.  These representations are determined by
the Mackey obstructions which can vary wildly even in the seemingly
most benign situations.  This wild behavior is often an insurmountable
obstruction to describing the Jacobson topology.  In the case where
one assumes away the difficulty --- that is, where we assume all the
Mackey obstructions vanish --- there has been considerable progress if
it is also assumed that the stabilizers are constant, see, for example
\cite{raewil:iumj91,olerae:jfa90,raeros:tams88,phirae:jot84}, and when
the stabilizer map is continuous \cite{raewil:jfa88}.  When nontrivial
Mackey obstructions are allowed, the progress has been more modest and
is usually accompanied with robust hypotheses \cite{ech:ma92,
  ech:mams96, echros:pjm95,echwil:ma95,horr:etds86}.  In this article,
we take on the case where $G$ acts properly on the spectrum $X$ of a
continuous-trace \cs-algebra $B$.  This in particular guarantees that
the Mackey-Rieffel-Green machine works very nicely and gives us a tidy
set-theoretic description of $\spec{B\rtimes_{\beta}G}$.
Specifically, if $[\omega_{x}]\in H^{2}(G_{x},\T)$ is the Mackey
obstruction at $x\in X$, then we can form the set
\begin{equation*}
  \stabxbh=\set{(x,\sigma) :\text{$x\in X$ and $\sigma\in\widehat
      G_{x,\omega_{x}}$}}, 
\end{equation*} 
where $\widehat G_{x,\omega_{x}}$ is the collection of irreducible
$\omega_{x}$-representations of the stability group $G_{x}$.  Then
$\stabxbh$ carries a natural $G$-action and we can identify the orbit
space $G\backslash \stabxbh$ with the spectrum of the crossed product
via an induction process: $(x,\sigma)\mapsto \ind_{x}(\sigma)$.  Our
first result is to equip $\stabxbh$ with a natural topology so that
induction induces a homeomorphism with the spectrum.

However to give a more useful and concrete description of the topology
on \stabxbh, and hence on the spectrum of the crossed product, we
require some hypotheses on $G$ (to at least partially tame the Mackey
obstructions).  In our main results, we are able to give such
descriptions in the cases where (1)~$G$ is discrete, and (2)~when $X$
is a manifold and $G$ is a Lie group acting smoothly, as well as
properly, on $X$.  For example, in the case where $G$ is discrete, we
can show that $(x_{n},\sigma_{n})\to (x,\sigma)$ in \stabxbh\ if and
only if $x_{n}\to x$ in $X$ and we eventually have $G_{x_{n}}\subseteq
G_{x}$ and $\sigma_{n}$ equivalent to a subrepresentation of
$\sigma\restr {G_{x_{n}}}$.  (The result for Lie groups is similar,
but a bit more technical.)

Since the methods of proof in our main results seem to require
separability, we assume throughout that the spaces and groups that
appear are second countable and that our algebras are separable.
Representations and other homomorphisms between \cs-algebras are
assumed to be $*$-preserving.

The starting point for this article is a nice structure theorem, in
terms of a generalized fixed point algebras, for crossed products for
proper actions {on spaces} due to the first author and H.~Emerson
\cite{echeme:em11}.  In fact, we make considerable use of some of the
results in that paper as well as some ideas from the Ph.D.~thesis of
Katharina Neumann \cite{neu:phd11}.  In \S1, we review the
construction of the generalized fixed point algebra from
\cite{echeme:em11} as well as some preliminary results on \cs-bundles.
In \S2, we specialize to the case of {strictly} proper actions on
continuous-trace \cs-algebras and study in detail the fixed point
algebras that arise in our investigations.  We also state the
particular version of the Mackey-Rieffel-Green machine that we will
employ.  In \S3, we introduce the topology on \stabxbh, and show that
the spectrum of $B\rtimes_{\beta}G$ is the quotient.  Our main results
are stated and proved in \S4.  In \S5, we exhibit a class of
  group extensions --- namely compact extensions $G$ of abelian groups
  --- where we can apply our results to the spectrum of $G$.  In
  particular, we work out in detail the spectrum of the
  crystallographic group \textbf{p4g}. Even in this case, our
  description is very subtle and clearly demonstrates the difficulty
  of working with varying Mackey obstructions.

\section{The Bundle Structure of a $\xtg$-algebra}
\label{sec:bundle-structure-xtg}

Recall that we call a \cs-algebra {$B$} a \emph{$C_{0}(X)$-algebra} if
there is a nondegenerate map $\Phi:C_{0}(X)\to ZM(B)$.  For the basics
on $C_{0}(X)$-algebras and their associated \usc\ \cs-bundles, we
refer to \cite{wil:crossed}*{Appendix~C}.  For convenience, we recall
some standard facts and notations here.  In particular, every
$C_{0}(X)$-algebra $B$ is $C_{0}(X)$-isomorphic to the section algebra
$\sa_{0}(X,\B)$ of an \usc\ \cs-bundle $p:\B\to X$.  The fibre over
$x$ can be identified with the quotient $B(x)=B/I_{x}$, where $I_{x}$
is the ideal of $B$ generated by products $f\cdot a$ with $f\in
C_{0}(X)$ such that $f(x)=0$ (and where, as is standard, we have
written $f\cdot a$ in place of $\Phi(f)(a)$).  When the map $b\mapsto
\|b\|$ is continuous from $\B$ to $\R$ (or equivalently, when the map
$x\mapsto \|b(x)\|$ is continuous for $b\in B$), we say that $\B$ is a
\emph{continuous} \cs-bundle.  If $Y$ is a closed subset of $X$ and if
$B=\sa_{0}(X,\B)$ is a $C_{0}(X)$-algebra, then
$B_{Y}:=\sa_{0}(Y,\B\restr Y)$ is the restriction of $B$ to $Y$.  We
can also realize $B_{Y}$ as the quotient of $B$ by the ideal generated
by $C_{0}(X\setminus Y)\cdot B$.

As in \cite{echeme:em11}, if $B$ is a $C_{0}(X)$-algebra for a
$G$-space $X$, and if $\beta:G\to \Aut B$ is a $C_{0}(X)$-linear
dynamical system, then we call $B$ a \emph{\xtg-algebra} if the
structure map $\Phi:C_{0}(X)\to ZM(B)$ is $G$-equivariant.

In this paper, we will always assume that the action of $G$ on $X$ is
\emph{proper} in that the map $(s,x)\mapsto (s\cdot x, x)$ is a proper
map from $G\times X$ to $X\times X$.  If $X$ is
a proper $G$-space, then the orbit space $\gux$ is Hausdorff and that
the stability groups $G_{x}=\set{s\in G:s\cdot x=x}$ are
compact \cite{pal:aom61}*{Theorem~1.2.9 and Proposition~1.1.4}.

If $B$ is a \xtg-algebra for a {strictly} proper $G$-action --- in
which case we say that $B$ is a \emph{{strictly} proper} \xtg-algebra
--- then there is an associated \emph{generalized fixed point algebra}
$B^{G,\beta}$ which agrees with the usual fixed point algebra when $G$
is compact (\cite{echeme:em11}*{\S2}).  We sketch the details of its
construction here.  We first form the \emph{induced algebra}
$C_{0}(X\times_{G,\beta}B)$ which consists of continuous functions
$F:X\to B$ which have the property $F(s\cdot
x)=\beta_{s}\bigl(F(x)\bigr)$ for all $s\in G$ and $x\in X$ and such
that $G\cdot x\mapsto \|F(x)\|$ vanishes at infinity on $\gux$.
(Induced algebras have been studied extensively for free and proper
actions of $G$ on $X$ --- for example see \cite{raewil:tams85},
\cite{rae:ma88} and \cite{hrw:tams00} --- but other than
\cite{echeme:em11}*{Lemma~2.2}, we are not aware of any systematic
study of induced algebras for {strictly} proper actions that are not
necessarily free.)  The properness of the $G$-action on $X$ guarantees
that the diagonal action of $G$ on $X\times X$ is proper.  Hence the
orbit space $G\backslash (X\times X)$ is Hausdorff and we obtain a
nondegenerate homomorphism
\begin{equation*}
  \Phi:C_{0}\bigl(G\backslash(X\times X)\bigr) \to
  ZM\bigl(C_{0}(X\times_{G,\beta}B) \bigr) 
\end{equation*}
by $\Phi(f)(F)\bigl(x)(y)=f([x,y])F(x)(y)$.  Then the generalized
fixed point algebra $B^{G,\beta}$ is defined to be the restriction of
the $C_{0}(G\backslash (X\times X))$-algebra
$C_{0}(X\times_{G,\beta}B)$ to the closed subspace $G\backslash
\Delta(X)$.  After identifying $X$ with $\Delta(X)$, we see that
$B^{G,\beta}$ inherits a $C_{0}(\gux)$-algebra structure.  As
mentioned above, if $G$ is compact, then the generalized fixed point
algebra $B^{G,\beta}$ is just the usual fixed point algebra.  Also, if
$B=C_{0}(X,A)$ for some \cs-algebra $A$ admitting a $G$-action
$\beta$, then $C_{0}(X,A)^{G,\lt\tensor\beta}$ is just the induced
algebra $C_{0}(X\times_{G,\beta}A)$.

If $G_{x}$ is the stabilizer of $x$, then $I_{x}=C_{0}(X\setminus\sset
x)\cdot B$ is a $G_{x}$-invariant ideal of $B$ and we have an induced
action $\beta^{x}:G_{x}\to \Aut B(x)$ given by
$\beta^{x}_{s}(b(x))=\beta_{s}(b)(x)$.  In fact, if $B=\sa_{0}(X,\B)$,
then $G$ acts on $\B$ via $\beta_{x,s}:B(x)\to B(s\cdot x)$ given by
\begin{equation}
  \label{eq:1}
  \beta_{x,s}\bigl(b(x)\bigr) =\beta_{s}(b)(s\cdot x)\quad\text{for
    all $b\in B$.}
\end{equation}
Then
\begin{equation}
  \label{eq:2}
  \beta_{s\cdot x,h}\circ \beta_{s,x}=\beta_{x,sh}\quad\text{for all
    $s,h\in G$ and $x\in X$.}
\end{equation}
If $h\in G_{x}$, then $\beta_{x,h}$ is just the induced automorphism
$\beta^{x}_{h}$ as defined above.\footnote{The $\beta_{x,s}$ give an
  action of the transformation groupoid \xtg\ on $\B$ which explains
  our terminology and notation for \xtg-algebras.  We will not make
  use of groupoid technology in this paper.}

\begin{lemma}
  \label{lem-genfixedalg}
  Let $B$ be a {strictly} proper \xtg-algebra.  Then the fibre over
  $G\cdot x$ of the $C_{0}(\gux)$-algebra $B^{G,\beta}$ is isomorphic
  to $B(x)^{G_{x},\beta^{x}}$.  In fact, we can realize $B^{G,\beta}$
  as
  \begin{multline}\label{eq:3}
    \bigl\{\,b\in\sa_{b}(X,\B):\text{$b(s\cdot
      x)=\beta^{x}_{s}\bigl(b(x)\bigr)$ for all $s\in G$ and $x\in X$,
      and }\\ \text{such that $G\cdot x\mapsto \|b(x)\|$ is in
      $C_{0}(\gux)$}\, \bigr\}.
  \end{multline}
  Then the isomorphism of the fibre of $B^{G,\beta}$ over $G\cdot x$
  with $B(x)^{G_{x},\beta^{x}}$ is given via evaluation at $x$.  If
  $\B$ is a continuous \cs-bundle, then $B^{G,\beta}$ is the section
  algebra of a continuous \cs-bundle over $\gux$.
\end{lemma}
\begin{proof}
  {By \cite[Lemma 2.2]{echeme:em11}, the fibre
    $C_{0}(X\times_{G,\beta}B)(G\cdot x)$ is isomorphic to the fixed
    point algebra $B^{G_{x},\beta}$ via evaluation at $x$.  On the
    other hand, the $(G\cdot x)$-fibre of $B^{G,\beta}$ is
    $\set{F(x)(x):F\in C_{0}(X\times_{G,\beta}B)}$.  Since
    $B^{G_{x},\beta}= \set{F(x):F\in C_{0}(X\times_{G,\beta}B)}$, we
    just need to see that evaluation at $x$ takes $B^{G_{x},\beta}$
    onto $B(x)^{G_{x},\beta^{x}}$.  (It clearly takes values in
    $B(x)^{G_{x},\beta^{x}}$.)  However, if $b_{x}\in
    B(x)^{G_{x},\beta^{x}}$, then there is a $b\in B$ such that
    $b(x)=b_{x}$.  But if $s\in G_{x}$, then
    $\beta_{s}(b)(x)=(\beta^{x})_{g}(b(x))=b_{x}$.  Thus we can let
    $\tilde b=\int_{G_{x}}\beta_{s}(b)\,ds$.  Then $\tilde b\in
    B^{G_{x},\beta}$ and $\tilde b(x)=b_{x}$.

    Clearly, we can identify $C_{0}(X\times_{G,\beta}B)$ with
    continuous sections of the \cs-bundle $X\times\B$ over $X\times X$
    such that
    \begin{equation}
      \label{eq:4}
      F(s\cdot x,s\cdot y)=\beta_{s}\bigl(F(x,\cdot)\bigr)(s\cdot y)
      =\beta_{y,s}\bigl(F(x,y)\bigr) \quad\text{and such that}
    \end{equation}
    $G\cdot x\mapsto \sup_{y}\|F(x,y)\|$ vanishes at infinity on
    $\gux$.

    On the other hand, \eqref{eq:3} is a closed
    $C_{0}(\gux)$-subalgebra $D$ of $\sa_{b}(X,\B)$.  The map sending
    $F$ to $(x\mapsto F(x,x))$ defines a homomorphism of
    $C_{0}(X\times_{G,\beta}B)$ into $D$ which factors through a
    $C_{0}(\gux)$-homomorphism $\psi$ of $B^{G,\beta}$ into $D$.  Then
    $\psi$ induces a map of the fibre $B^{G,\beta}(G\cdot x)$ into
    $D(G\cdot x)$.  Evaluation at $x$ clearly identifies $D(G\cdot x)$
    with a subalgebra of $B(x)^{G_{x},\beta^{x}}$, and the first part
    of the lemma implies the image is all of $B(x)^{G_{x},\beta^{x}}$
    and that $\psi$ induces an isomorphism of the fibres over $G\cdot
    x$.  Hence $\psi$ is an isomorphism so that we can identify $D$
    with $B^{G,\beta}$ as claimed.

    To verify the statement about continuous bundles, it suffices to
    see that $G\cdot x \mapsto \|b(x)\|$ is continuous for $b$ in
    \eqref{eq:3}.  But this is automatic if $\B$ is a continuous
    bundle.}
\end{proof}

If $B$ is a {strictly} proper \xtg-algebra, then so is $B\tensor
\kltg$ with $C_{0}(X)$ acting on the first factor and with respect to
the diagonal {$G$-}action $\beta\tensor\Ad\rho$, where $\rho:G\to
U\bigl(L^{2}(G)\bigr)$ is the right regular representation.  Clearly,
$B\tensor \kltg$ is the section algebra {of} a bundle $\B\tensor
\kltg$ over $X$ built from the fibres $B(x)\tensor \kltg$ and for
which $x\mapsto b(x)\tensor T$ is a continuous section for every $b\in
B$ {and $T\in \kltg$}.  Since we can identify $\Prim B$ and $\Prim
(B\tensor\kltg)$, it follows from \cite{kirwas:ma95} or
\cite{wil:crossed}*{Theorem~C.26} that $\B\tensor\kltg$ is continuous
when $\B$ is.  Hence we can quote \cite{echeme:em11}*{Theorem~2.14} to
deduce the following.
\begin{thm}[\cite{echeme:em11}*{Theorem~2.14}]
  \label{thm-EE2.14}
  Let $B$ be a {strictly} proper \xtg-algebra.  Then
  \mbox{$B\rtimes_{\beta}G$} is isomorphic to the generalized fixed
  point algebra $\bigl(B\tensor\kltg\bigr)^{G,\beta\tensor\Ad\rho}$,
  which is a $C_{0}(\gux)$-algebra with fibres over $G\cdot x$
  isomorphic to $\bigl(B(x)\tensor
  \kltg\bigr)^{G_{x},\beta^{x}\tensor\Ad\rho}$.  Moreover if $B$ is
  (the section algebra of a) continuous bundle of \cs-algebras, then
  so is $\bigl(B\tensor\kltg\bigr)^{G,\beta\tensor\Ad\rho}$.
\end{thm}

\section{Proper Actions on \cs-Alebras with Continuous Trace}
\label{sec:proper-actions-cs}

Now we restrict to the situation where $B$ is a continuous-trace
\cs-algebra with spectrum $X$.  Then we may assume that $B$ is the
section algebra $\sa_{0}(X,\B)$ of a \emph{continuous} \cs-bundle $\B$
with fibres $B(x)=\K(\H_{x})$ for Hilbert spaces $\H_{x}$.  We will
also assume that we have a dynamical system $\beta:G\to \Aut B$ such
that the induced action of $G$ on $X$ is proper.  In the sequel, we
will simply say that \emph{$G$ acts {strictly} properly on $B$}.  We
use the term ``strictly properly'' to make sure that our notion is not
mistaken with the much weaker notion of proper actions on C*-algebras
due to Rieffel  \cite{rie:pm88,rie:em04}.

If $\H$ is a Hilbert space, then $U(\H)$ will denote the group of
unitary operators on $\H$ endowed with the strong operator topology.
Then $\Aut \K(\H)$ can be identified with the projective unitary group
$U(\H)/\T1$ {in} a canonical way, and we can view
$\beta^{x}:G_{x}\to\Aut B(x)$ as a continuous homomorphism of $G_{x}$
into $PU(\H_{x})$.  (Thus, $\beta^{x}$ is often called a
\emph{projective representation} of $G_{x}$ on $\H_{x}$.\footnote{For
  an elementary discussion of projective representations, see
  \cite{wil:crossed}*{\S D.2}.}  Since each $\H_{x}$ is separable,
there is a Borel map $V_{x}:G_{x}\to U(\H_{x})$ such that
$\beta^{x}=\Ad V_{x}$.  Our choice of $V_{x}$ determines a Borel
$2$-cocycle $\omega_{x}:G\times G\to \T$ such that
\begin{equation}
  \label{eq:5}
  V_{x}(s)V_{x}(t)=\omega_{x}(s,t)V_{x}(st)\quad\text{for all $s,t\in
    G_{x}$.} 
\end{equation}
Assuming, as we do, that we have arranged that
${V_{x}}(e)=1_{\H_{x}}$, then our cocycle is \emph{normalized} in that
$\omega_{x}(e,s)=1=\omega_{x}(s,e)$ for all $s\in G$.  A normalized
2-cocycle is called a \emph{multiplier} on $G$.  The class
$[\omega_{x}]\in H^{2}(G_{x}\,T)$ depends only on $\beta_{x}$ and
represents the obstruction to $\beta_{x}$ being implemented by a
unitary representation of $G_{x}$ on $\H_{x}$.  It is called the
\emph{Mackey obstruction} at $x$.

Our task in this section is to examine the structure of the fibres
$\bigl(B(x) \tensor {\kltg}\bigr)^{G_{x},\beta^{x}\tensor\Ad\rho}$ of
$\bigl(B\tensor\kltg\bigr)^{G,\beta\tensor\Ad\rho}$. Then if
$B(x)=\K(\H_{x})$ and $\beta^{x}=\Ad V_{x}$ as above, we want to
examine $ \K\bigl({\H_{x}}\tensor
L^{2}(G)\bigr)^{G_{x},\Ad(V_{x}\tensor \rho)}$.  Thus, dropping the
``$x$'', we want to consider the fixed point algebra
\begin{equation*}
  \K\bigl({\H}\tensor L^{2}(G)\bigr) ^{K,\Ad(V\tensor\rho)}
\end{equation*}
where $K$ is a compact subgroup of $G$, and where $V:K\to U(\H)$ is a
Borel lift of a projective representation $\beta:K\to PU(\H)$.  We let
$\omega\in H^{2}(K,\T)$ be the multiplier associated to $V$ as in
\eqref{eq:5}.

A Borel map $\sigma:K\to U(\H_{\sigma})$ which satisfies
$\sigma(s)\sigma(t)=\omega(s,t)\sigma(st)$ is called {\emph{a
    multiplier representation with multiplier $\omega$} or, more
  simply, an \emph{$\omega$-representation}} of $K$. We have two
canonical $\omega$-representations of $K$ on $L^{2}(K)$, called the
left- and right-regular $\omega$-representations, which are given,
respectively, by
\begin{equation*}
  \lambda_{K}^{\omega}(s)\xi(t)=\omega(s,s^{-1}t)\xi(s^{-1}t)\quad\text{and}
  \quad \rho_{K}^{\omega}(s)\xi(t)=\omega(t,s)\xi(ts).
\end{equation*}
The multiplier $\omega$ determines {a} group structure on $K\times\T$:
\begin{equation*}
  (s,z)(t,w)=(st,\omega(s,t)zw)\quad\text{and}\quad(s,z)^{-1}=(s^{-1},
  \overline{z\omega(s,s^{-1})}). 
\end{equation*}
A nontrivial result due to Mackey implies that the resulting group,
denoted by $K\times_{\omega}\T$, has a compact topology {(since $K$ is
  compact)} such that $K\times_{\omega}\T$ is a central extension of
$K$ by $\T$. (For references and more details, see page~376 of
\cite{wil:crossed}.)  Furthermore, the $\omega$-representations of $K$
are then in one-to-one correspondence with the unitary representations
of $K\times_{\omega}\T$ whose restriction to $\T$ is a multiple of the
identity character $z\mapsto z$ (see, for example,
\cite{wil:crossed}*{Proposition~D.28}).
In a similar way, we get a correspondence between the
$\bar\omega$-representations of $K$ and the unitary representations of
$K\times_{\omega}\T$ that restrict to a multiple of the character
$z\mapsto \bar z$ on $\T$.  Since $K\times_{\omega}\T$ is compact,
every irreducible $\omega$-representation of $K$ is
finite-dimensional.  We will write $\widehat K_{\omega}$ for the set
of irreducible $\omega$-representations.

\begin{remark}
  \label{rem-basic-omega}
  Using the above correspondence and the well-known theory of unitary
  representations of compact groups (for example, see
  \cite{deiech:principles09}*{\S\S7.2--3}), we see that the space
  $\H_{W}$ of any $\omega$-representation $W:K\to U(\H_{W})$ on a
  separable Hilbert space decomposes as
  \begin{equation*}
    \H_{W}=\bigoplus_{\sigma\in\widehat K_{\omega}}\H(\sigma),
  \end{equation*}
  where $\H(\sigma)$ is the \emph{isotope} of $\sigma$.  That is,
  $\H(\sigma)$ is the union of closed $K$-invariant subspaces $\H'$ of
  $\H_{W}$ such that the restriction of $W$ to $\H'$ is equivalent to
  $\sigma$.  Then each $\H(\sigma)$ decomposes as a tensor product
  $\H'_{\sigma} \tensor \H_{\sigma}$ in such a way that
  \begin{equation*}
    \H_{W}\cong \bigoplus_{\sigma\in\widehat K_{\omega}}\H'_{\sigma}\tensor
    \H_{\sigma} \quad\text{and}\quad W\cong \bigoplus_{\sigma\in \widehat
      K_{\sigma}} 1_{\H'_{\sigma}}\tensor \sigma.
  \end{equation*}
  Then, and this is the point, the fixed point algebra is given by
  \begin{equation*}
    \K(\H_{W})^{K,\Ad W} \cong \bigoplus_{\sigma\in\widehat K_{\omega}}
    \K(\H'_{\sigma}) \tensor 1_{\H_{\sigma}}.
  \end{equation*}
\end{remark}

Therefore to understand the structure of $\K\bigl(\H\tensor
L^{2}(G)\bigr)^{K,\Ad V\tensor\rho}$, we need to decompose the Hilbert
space $\H\tensor L^{2}(G)$ into isotopes as above via an analogue of
the Peter-Weyl Theorem for compact groups.  Although the result is
certainly well-known, we include a statement and proof as we lack a
direct reference.  First, for any $\omega$-representation $\sigma:K\to
U(\H_{\sigma})$, we write $\H_{\sigma}^{*}$ for the conjugate Hilbert
space and let $\sigma^{*}:K\to U(\H_{\sigma}^{*})$ be the conjugate
representation of $\sigma$:
${\sigma^{*}(s)}(v^{*})=\bigl(\sigma(s)(v)\bigr)^{*}$.  Then
$\sigma^{*}$ is a $\bar\omega$-representation and the map
$\sigma\mapsto\sigma^{*}$ is a bijection between $\widehat K_{\omega}$
and $\widehat K_{\bar\omega}$.  There is a unique linear map
\begin{equation}
  \label{eq:7}
  \H_{\sigma}\tensor\H_{\sigma}^{*}\to L^{2}(K),
\end{equation}
sending the elementary tensor $v\tensor w^{*}$ to $g^{\sigma}_{v,w}$,
where $g_{v,w}^{\sigma}(k)=\sqrt{d_{\sigma}}\bip(v|\sigma(k)w)$ and
$d_{\sigma}=\dim\H_{\sigma}$.

\begin{lemma}[Peter-Weyl]
  \label{lem-peter-weyl}
  Let $\omega$ be a multiplier on the compact group $K$.  Then the
  maps \eqref{eq:7} above induce a Hilbert-space isomorphism of
  \begin{equation*}
    \bigoplus_{\sigma\in \widehat K_{\omega}}
    \H_{\sigma}\tensor\H_{\sigma}^{*} \quad\text{with}\quad L^{2}(K)
  \end{equation*}
  which intertwines the representation $\bigoplus_{\sigma}
  \sigma\tensor 1_{\H_{\sigma}^{*}} $ with the left regular
  $\omega$-representation $\lambda_{K}^{\omega}$, as well as the
  representation $\bigoplus_{\sigma} 1_{\H_{\sigma}}\tensor
  \sigma^{*}$ with the right regular
  $\bar\omega$-representation~$\rho_{K}^{\bar\omega}$.
\end{lemma}
\begin{proof}
  As we will see, the result follows readily from the classical
  Peter-Weyl Theorem for the compact group $K\times_{\omega}\T$ which
  gives us the decomposition
  \begin{equation*}
    L^{2}(K\times_{\omega}\T)\cong
    \bigoplus_{\tau\in\spec{K\times_{\omega}\T}} \H_{\tau}\tensor\H_{\tau}^{*}
  \end{equation*}
  in which an elementary tensor $v\tensor w^{*}$ in $\H_{\tau}\tensor
  \H_{\tau}^{*}$ corresponds to the function $f_{v,w}^{\tau}$ given by
  \begin{equation*}
    f_{v,w}^{\tau}(k,z):=\sqrt{d_{\tau}}\bip(v|\tau(k,z)w).
  \end{equation*}
  Since $\T$ is central in $K\times_{\omega}\T$, the restriction to
  $\T$ of any irreducible representation $\tau$ of
  $K\times_{\omega}\T$ must be a multiple of a character
  $\chi_{n}\in\widehat\T$ of the form $\chi_{n}(z)=z^{n}$ (with
  $n\in\Z$).  Thus
  $\tau(k,z)=\tau\bigl((e,z)(k,1)\bigr)=\chi_{n}(z)\tau(k,1)$, and
  \begin{equation*}
    f_{v,w}^{\tau}(k,z)=\bar\chi_{n}(z)g^{\tau}_{v,w}(k)\quad\text{where}\quad
    g^{\tau}_{v,w}(k):=\sqrt{d_{\tau}}\bip(v|\tau(k,1)w).
  \end{equation*}
  Recall that if $\set{v^{\tau}_{1},\dots,v^{\tau}_{n_{\tau}}}$ is a
  basis for $\H_{\tau}$ if we let
  $f^{\tau}_{ij}=f^{\tau}_{v_{i},v_{j}}$, then
  \begin{equation*}
    \set{f^{\tau}_{ij}:\text{$\tau\in\spec{K\times_{\omega}\T}$ and
        $1\le i,j\le n_{\tau}$}}
  \end{equation*}
  is an orthonormal basis for $L^{2}(K\times_{\omega}\T)$.  Thus if we
  let $g_{ij}^{\tau}(k)=f_{ij}^{\tau}(k,1)$ and let
  $\spec{K\times_{\omega}\T}_{n}$ be the collection of $\tau$
  restricting to a multiple of $\chi_{n}$ on $\T$, then we get an
  orthonormal basis of $L^{2}(K)\tensor L^{2}(\T) \cong
  L^{2}(K\times_{\omega}{\T})$ of the form
  \begin{equation*}
    \bigcup_{n\in\Z}\set{g_{ij}^{\tau}\tensor\bar\chi_{n}:\text{$\tau\in
        \spec{K\times_{\omega}\tau}_{n}$ and $1\le i,j\le n_{\tau}$}}.
  \end{equation*}
  Since $\set{\bar\chi_{n}:n\in\Z}$ is an orthonormal basis for
  $L^{2}(\T)$, we can conclude that for each $n\in\Z$,
  \begin{equation*}
    \set{g_{ij}^{\tau}:\text{$\tau\in\spec{K\times_{\omega}\T}_{n}$ and
        $1\le i,j\le n_{\tau}$}}
  \end{equation*}
  is an orthonormal basis for $L^{2}(K)$.  In the case, $n=1$ we get a
  decomposition as in the lemma.

  We still need to see that this isomorphism intertwines the given
  representations.  But the computation
  \begin{align*}
    \bigl(\lambda_{K}^{\omega}(l)g^{\sigma}_{v,w}\bigr) (k) &=
    \omega(l,l^{-1}k)g_{v,w} ^{\sigma}(l^{-1}k) \\
    &= \omega(l,l^{-1}k) \sqrt{d_{\sigma}} \bip(v|\sigma(l^{-1}k)w) \\
    \intertext{which, since $\sigma$ is an $\omega$-representation,
      is} &= \omega(l,l^{-1}k)\omega(l^{-1},k) \sqrt{d_{\sigma}}
    \bip(v|\sigma(l^{-1}) \sigma(k) w) \\
    \intertext{which, since
      $\sigma(l^{-1})=\omega(l,l^{-1})\sigma(l)^{*}$, is} &=
    \omega(l,l^{-1}k)\omega(l^{-1},k) \bar\omega(l,l^{-1})
    \sqrt{d_{\sigma}} \bip(v|\sigma(l)^{*} \sigma(k) w) \\
    \intertext{which, since the cocycle identity implies
      $\omega(l,l^{-1}k)\omega(l^{-1},k)
      \bar\omega(l,l^{-1})=\omega(e,k)=1$, is}
    &= \sqrt{d_{\sigma}} \bip(\sigma(l)v|{\sigma(k)w})\\
    &= g_{\sigma(l)v,w}^{\sigma}(k)
  \end{align*}
  shows that this isomorphism intertwines $\bigoplus_{\sigma}
  \sigma\tensor 1_{\H_{\sigma}^{*}}$ and the left-regular
  $\omega$-representation $\lambda_{K}^{\omega}$.  The computation
  involving the right-regular $\omega$ representation is similar, but
  less messy.
\end{proof}

\begin{remark}\label{rem-switch}
  We can apply the above lemma to $\bar\omega$ and then use the
  bijection between $\widehat K_{\omega}$ and $\widehat
  K_{\bar\omega}$ to obtain a decomposition $L^{2}(K)$ with
  $\bigoplus_{\sigma\in \widehat K_{\omega}} \H_{\sigma}^{*}\tensor
  \H_{\sigma}$ which intertwines $\bigoplus _{\sigma}
  \sigma^{*}\tensor 1_{\H_{\sigma}}$ with $\lambda_{K}^{\bar\omega}$
  and $\bigoplus_{\sigma} 1_{\H_{\sigma}^{*}}\tensor \sigma$ with
  $\rho_{K}^{\omega}$.  This is the decomposition we will employ
  below.
\end{remark}

We will need a twisted version of the absorption principle for regular
representations of $G$.  The proof is a straightforward calculation
which we omit.
\begin{lemma}
  \label{lem-absorbtion}
  Suppose that $\omega$ and $\mu$ are multipliers on a locally compact
  group $G$.  Let $V:G\to U(\H)$ be an $\omega$-representation of $G$
  and $W:L^{2}(G,\H)\to L^{2}(G,\H)$ the unitary operator given by
  $(W\xi)(s)=V_{s} \xi(s)$ for $\xi\in L^{2}(G,\H)$ and $s\in G$.
  Then for all $s\in G$,
  \begin{equation*}
    W(V\tensor \rho_{G}^{\mu})(s)W^{*} =
    1_{\H}\tensor\rho_{G}^{\omega\mu}(s) \quad\text{and}\quad
    {W^*}(V\tensor \lambda_{G}^{\mu})(s)W=1_{\H}\tensor
    \lambda_{G}^{\omega\mu}(s). 
  \end{equation*}
  In particular, $V\tensor \rho_{G}\cong 1_{\H}\tensor
  \rho_{G}^{\omega}$ and $V\tensor \lambda_{G}\cong 1_{H}\tensor
  \lambda_{G}^{\omega}$.
\end{lemma}

Suppose now that $K$ is a compact subgroup of a locally compact group
$G$.  If $\tau$ is a unitary representation of $K$ on $\H_{\tau}$,
then the induced representation $U^{\tau}$ acts by left-translation on
the Hilbert space $L^{2}(G\times_{K,\tau}\H)$ of (almost everywhere
equivalence classes of) square integrable functions from $G$ to
$\H_{\tau}$ satisfying $\xi(sk)=\tau(k^{-1})\bigl(\xi(s)\bigr)$ for
all $s\in G$ and $k\in K$.  Alternatively, we can realize
$L^{2}(G\times_{G,\tau}\H)$ as $L^{2}$-sections of the vector bundle
$K\backslash (G\times\H_{\tau})$ for the diagonal action $k\cdot
(g,v)=(gk^{-1},\tau(k)(v))$.  In particular, the isomorphism
$U:L^{2}(G)\to L^{2}(G\times_{G,\lambda_{K}}L^{2}(K))$, given by
$(U\xi)(s)(k)=\xi(sk)$ for $s\in G$ and $k\in K$, intertwines the
left-regular representation $\lambda_{G}$ of $G$ with
$U^{\lambda_{K}}$.  It also intertwines the restriction of the
right-regular representation $\rho_{G}\restr K$ with the pointwise
action of the {right regular representation} $\rho_{K}$ {of $K$} on
the elements of $L^{2}(G\times_{G,\lambda_{K}}L^{2}(K))$.

Thus if $V:K\to U(\H)$ is an $\omega$-representation, we get a unitary
\begin{equation*}
  \tilde U:\H\tensor L^{2}(G)\to L^{2}\big(G\times_{K,1_{\H}\tensor
    \lambda_{K}}(\H\tensor L^{2}(K))\big)
\end{equation*}
which intertwines the representation $1_{\H}\tensor \lambda_{G}$ of
$G$ on $\H\tensor L^{2}(G)$ with the induced representation
$U^{1_{\H}\tensor \lambda_{K}}$, as well as the representation
$V\tensor \rho_{G}\restr K$ with the representation of $K$ given by
the pointwise action of $V\tensor \rho_{K}$ on elements of
$L^{2}\big(G\times_{K,1_{\H}\tensor \lambda_{K}}{(\H\otimes
  L^{2}(K))}\big)$.

Let $W$ be the unitary on $L^{2}(K,\H)$ from
Lemma~\ref{lem-absorbtion} which intertwines $1_{\H}\tensor
\lambda_{K}$ with $V\tensor\lambda_{K}^{\bar\omega}$ and $V\tensor
\rho_{K}$ with $1_{\H}\tensor \rho_{K}^{\omega}$.  Then, after
identifying $\H\tensor L^{2}(K)$ with $L^{2}(K,\H)$, we can apply $W$
pointwise to elements in $L^{2}(G\times_{K,1_{\H}\tensor
  \lambda_{K}}(\H\tensor L^{2}(K)))$ to obtain a unitary
\begin{equation*}
  \tilde W:L^{2}(G\times_{K,1_{\H}\tensor\lambda_{K}}(\H\tensor
  L^{2}(K))) \to L^{2}(G\times_{K,V\tensor
    \lambda_{K}^{\bar\omega}}(\H\tensor L^{2}(K)))
\end{equation*}
which intertwines the induced representations $U^{1_{\H}\tensor
  \lambda_{K}}$ and $U^{V\tensor \lambda_{K}^{\bar\omega}}$ of $G$,
and the pointwise acting $\omega$-representation $V\tensor \rho_{K}$
with the pointwise acting $\omega$-representation $1\tensor
\rho_{K}^{\omega}$.

Now we plug in the Peter-Weyl decomposition $L^{2}(K)\cong
\bigoplus_{\sigma\in \widehat K_{\omega}} \H_{\sigma}^{*}\tensor
\H_{\sigma}$ of Lemma~\ref{lem-peter-weyl} (using
Remark~\ref{rem-switch}) together with the descriptions of
$\lambda_{K}^{\bar\omega}$ and $\rho_{K}^{\omega}$ with respect to
this decomposition to obtain the following.
\begin{prop}
  \label{prop-decomp-1}
  Let $V$ be an $\omega$-representation of a {compact} subgroup $K$ of
  $G$ on $\H$.  Then there is a unitary
  \begin{equation*}
    \Phi:\H\tensor L^{2}(G)\to \bigoplus_{\sigma\in\widehat K_{\omega}}
    L^{2} (G\times_{K,V\tensor
      \sigma^{*}}(\H\tensor\H_{\sigma}^{*}))\tensor \H_{\sigma}
  \end{equation*}
  which intertwines the $\omega$-representations $V\tensor
  \rho_{G}\restr K$ of $K$ on $\H\tensor L^{2}(G)$ with the
  $\omega$-representation $\bigoplus_{\sigma\in\widehat K_{\omega}}
  1_{L^{2}(G\times_{K}\H_{\sigma}\tensor\H_{\sigma}^{*})}\tensor\sigma$.
  Moreover,
  \begin{equation*}
    \K(\H\tensor L^{2}(G))^{K,\Ad(V\tensor \rho_{G})} \cong
    \bigoplus_{\sigma\in\widehat K_{\omega}}
    \K(L^{2}(G\times_{K,V\tensor\sigma^{*}}(\H\tensor
    \H_{\sigma}^{*})))\tensor 1_{\H_{\sigma}}.
  \end{equation*}
\end{prop}
\begin{proof}
  The first assertion follows from the above discussion and the second
  from Remark~\ref{rem-basic-omega}.
\end{proof}

\begin{notation}
  Since the Hilbert spaces
  $L^{2}\bigl(G\times_{K,V\tensor\sigma^{*}}(\H\tensor\H_{\sigma}^{*})\bigr)$
  are going to be ubiquitous in the sequel, we are going to introduce
  {the} notation
  \begin{equation*}
    \ltgvs:=L^{2}\bigl(G\times_{K,V\tensor\sigma^{*}}(\H\tensor\H_{\sigma}^{*})\bigr)
  \end{equation*}
  in an attempt to make some of our formulas easier to parse.
\end{notation}

\begin{remark}
  \label{rem-summary}
  The above proposition shows that the fixed-point algebra
  $\K(\H\tensor L^{2}(G))^{K,\Ad(V\tensor\rho_{G})}$ decomposes into
  blocks of compact operators such that each $\sigma\in\widehat
  K_{\omega}$ provides the block $\K(\ltgvs) $ with multiplicity
  $d_{\sigma}:=\dim \H_{\sigma}$.  Therefore, as a \cs-algebra, the
  fixed-point algebra is isomorphic to $\bigoplus_{\sigma\in\widehat
    K_{\omega}} \K(\ltgvs)$.  Thus there is a bijection
  \begin{equation*}
    \ind_{K}:\widehat K_{\omega}\to \bigl(\K(\H\tensor
    L^{2}(G))^{K,\Ad(V\tensor\rho_{G})}\bigr)^{\wedge}
  \end{equation*}
  which sends $\sigma\in\widehat K_{\omega}$ to the projection
  \begin{equation*}
    \ind_{K}\sigma: \K(\H\tensor
    L^{2}(G))^{K,\Ad(V\tensor\rho_{G})} \to
    \K(\ltgvs).
  \end{equation*}
\end{remark}

\begin{remark}
  By composing with the natural map of $U(\H_{\sigma})$ onto
  $PU(\H_{\sigma})$ each $\omega$-representation $\sigma$ determines a
  projective representation.  If $\sigma'$ is an
  $\omega'$-representation representing the same projective
  representation, then $[\omega]=[\omega']\in H^{2}(K,\T)$.  Thus we
  can speak of the collection of equivalence classes $\widehat
  K_{[\omega]}$ of irreducible projective representations with class
  $[\omega]\in H^{2}(K,\T)$.  (Of course, there is an obvious
  bijection of $\widehat K_{\omega}$ with $\widehat K_{[\omega]}$.)
  Thus if $B:=\K(\H)$, then an action $\beta:K\to \Aut B$ determines a
  class $[\omega]\in H^{2}(K,\T)$ which does not depend on our choice
  of lift $V:K\to U(\H)$.  {Therefore}, the previous discussion says
  we have an isomorphism
  \begin{equation}
    \label{eq:9}
    \bigl(B\tensor \K(L^{2}(G))\bigr)^{\beta\tensor\Ad\rho_{G}} \cong
    \bigoplus_{\sigma\in \widehat K_{[\omega]}} \K_{\sigma}(\ltgvs),
  \end{equation}
  \emph{for any choice of lift} $V:K\to U(\H)$ for $\beta$. Moreover,
  as a subalgebra of $\K(\H\tensor L^{2}(G))$, each summand
  $\K_{\sigma}(\ltgvs)$ appears with multiplicity $d_{\sigma}$.

  In what follows, we call the summand $\K_{\sigma}(\ltgvs)$ in
  \eqref{eq:9} the summand of $\bigl(B\tensor
  \kltg\bigr)^{\beta\tensor\Ad \rho_{G}}$ \emph{of type $\sigma$}.
\end{remark}

As an immediate consequence of the above results (and conventions in
Remark~\ref{rem-summary}), we obtain the following.

\begin{prop}
  \label{prop-ind-gx}
  Suppose that $\beta:G\to \Aut B$ is a {strictly} proper action of
  $G$ on the continuous-trace \cs-algebra $B$ with spectrum $X$.  Let
  $[\omega_{x}]\in H^{2}(G_{x},\T)$ denote the Mackey obstruction of
  $\beta^{x}:G_{x}\to \Aut B(x)$.  Then there is a canonical bijection
  \begin{equation*}
    \ind_{x}:\widehat G_{x,[\omega_{x}]}\to \bigl(\bigl(B(x)\tensor
    \kltg\bigr)^{G_{x} , \beta^{x}\tensor \Ad\rho_{G}}\bigr)^{\wedge}.
  \end{equation*}
  Moreover, the diagram
  \begin{equation*}
    \xymatrix @C+3pc @R+1pc{\widehat G_{x,[\omega_{x}]}
      \ar[r]^-{\ind_{x}} \ar[d]_{C_{s}}& 
      \bigl(\bigl(B(x)\tensor
      \kltg\bigr)^{G_{x},\beta^{x}\tensor\Ad\rho_{G}} \bigr)^{\wedge}
      \ar@<-9.25ex>[d]^{(\beta_{x,s}\tensor\Ad\rho_{G}(s))\circ{}}\\ 
      \widehat G_{s\cdot x,[\omega_{s\cdot x}]} \ar[r]_-{\ind_{s\cdot x}} &
      \bigl(\bigl( B(s\cdot x)\tensor\kltg\bigr)^{G_{s\cdot x},\beta^{s\cdot
          x }
        \tensor\Ad\rho_{G}}\bigr)^{\wedge}
     }  
  \end{equation*}
  commutes, where $\beta_{x,s}:B(x)\to B(s\cdot x)$ is the isomorphism
  given in \eqref{eq:1} and $C_{s}$ is given by $\sigma\mapsto
  s\cdot\sigma$ with $s\cdot \sigma(k)=\sigma(s^{-1}ks)$.
\end{prop}
\begin{proof}
  The only issue is the commutativity of the diagram.  But
  \eqref{eq:2} implies that conjugation of
  $\beta^{x}\tensor\Ad\rho_{G}$ by the isomorphism $\beta_{x,s}\tensor
  \Ad\rho_{G}(s)$ gives the action $\beta^{s\cdot
    x}\tensor\Ad\rho_{G}$.  The rest follows from straightforward
  computations.
\end{proof}

Recalling that every irreducible representation of {a}
$C_{0}(Y)$-algebra $A$ factors through a fibre, so that $\widehat
A=\coprod_{y\in Y} \widehat A(y)$, Theorem~\ref{thm-EE2.14} and our
Proposition~\ref{prop-ind-gx} imply the following version of the
Mackey-Green-Rieffel machine.

\begin{thm}[Mackey-Green-Rieffel]
  \label{thm-mackey-machine}
  Suppose that $\beta:G\to\Aut B$ is a {strictly} proper action on the
  continuous-trace \cs-algebra with spectrum $X$.  For each $x\in X$,
  let $[\omega_{x}]\in H^{2}(G_{x},\T)$ be {the} Mackey obstruction
  for $\beta^{x}:G_{x}\to \Aut B(x)$ (where $\beta^{x}$ is the induced
  action of the stabilizer subgroup $G_{x}$ on the fibre $B(x)$).  Let
  \begin{equation*}
    \stabxbh:= \set{(x,\sigma):\text{$x\in X$ and
        $\sigma\in \widehat G_{x,[\omega_{x}]}$}},
  \end{equation*}
  and let $G$ act on $\stabxbh$ via $s\cdot (x,\sigma):=(s\cdot
  x,s\cdot \sigma)$ with $s\cdot \sigma(k)=\sigma(s^{-1}ks)$.  Then
  there is a surjective map
  \begin{equation*}
    \Ind:\stabxbh\to \spec{B\rtimes_{\beta}G} \cong
    \bigl(\bigl(B\tensor\kltg\bigr)^{G,\beta\tensor\Ad\rho_{G}}\bigr)^{\wedge},
  \end{equation*}
  given by sending $(x,\sigma)\in\stabxbh$ to the corresponding
  representation of the fibre
  $\bigl(B(x)\tensor\kltg\bigr)^{G_{x},\beta^{x}\tensor\Ad \rho_{G}}$
  (as in Proposition~\ref{prop-ind-gx}), which factors through a
  bijection of $G\backslash \stabxbh$ onto $\spec{B\rtimes_{\beta}G}$.
\end{thm}

\begin{remark}
  We should compare our version of the Mackey-Green-Rieffel machine
  with the classical approach.  There we start with the irreducible
  representation $\pi_{x}:B\to \K(\H_{x})$ (essentially evaluation at
  $x$), and let $V_{x}:G_{x}\to U(\H_{x})$ and $\omega_{x}\in
  Z^{2}(G_{x},\T)$ be as above.  Then if $\sigma\in \widehat
  G_{x,\omega_{x}}$ we obtain an irreducible unitary representation
  $(\pi_{x}\tensor 1_{\H_{\sigma}^{*}})\rtimes (V_{x}\tensor
  \sigma^{*})$ of $B\rtimes_{\beta}G_{x}$ on
  $\H_{x}\tensor\H_{\sigma}^{*}$.  Then we obtain an irreducible
  representation of $B\rtimes_{\beta}G$ via induction:
  $\Ind_{G_{x}}^{G} \bigl((\pi_{x}\tensor 1_{\H_{\sigma}^{*}})\rtimes
  (V_{x}\tensor \sigma^{*})\bigr)$.  Both this induced representation
  and our $\ind_{x}(\sigma)$ (as in Theorem~\ref{thm-mackey-machine})
  act on $L^{2}(G\times_{G_{x},V_{x}\tensor\sigma^{*}}(\H_{x}\tensor
  \H_{\sigma}^{*}))=\ltgvxs$.  It is not difficult to check that the
  two representations are the same.
\end{remark}

In order to understand the topology on $\spec{B\rtimes_{\beta}G}$, we
will need to compare $\bigl(B(x)\tensor
\kltg\bigr)^{G_{x},\beta^{x}\tensor\Ad \rho_{G}}$ with the fixed point
algebra $\bigl(B(x)\tensor\kltg\bigr)^{L,\beta^{x}\tensor\Ad\rho_{G}}$
for a closed subgroup $L$ of $G_{x}$.  To simplify the notation, we
consider the following set-up.  We suppose that $L$ is a closed
subgroup of a compact subgroup $K$ of $G$ and $\omega\in Z^{2}(K,\T)$.
We let $V:K\to U(\H)$ be an $\omega$-representation with $\beta:=\Ad
V$.  Then we may restrict everything to $L$ so that the fixed point
algebra $\K\bigl(\H\tensor L^{2}(G)\bigr)^{K,\Ad(V\tensor \rho_{G})}$
is a subset of $\K\bigl(\H\tensor L^{2}(G)\bigr)^{L,\Ad(V\tensor
  \rho_{G})}$.  Thus any block $\K(\ltgvs)$ of type $\sigma\in\widehat
K_{[\omega]}$ of the fixed-point algebra for the $K$-action must be
contained in a block $\K(L^{2}(G\times_{L,V\tensor
  \tau^{*}}(\H\times\H_{\tau}^{*})))=\K(\ltgvt)$ of type
$\tau\in\widehat L_{[\omega]}$ in the fixed-point algebra for the
$L$-action.  We aim to determine how many blocks of a given type
$\sigma$ line a block of type $\tau$.

Towards this end, we note that Proposition~\ref{prop-decomp-1} implies
that we can decompose $\H\tensor L^{2}(G)$ as
\begin{equation}\label{eq:10}
  \bigoplus_{\sigma\in\widehat K_{\omega}}
  \ltgvs\tensor\H_{\sigma}.
\end{equation}
Furthermore, the decomposition in \eqref{eq:10} is such that
$V\tensor\rho_{G}$ is intertwined with the $\omega$-representation
$\bigoplus_{\sigma} 1_{\ltgvs}\tensor\sigma$.  For each
$\sigma\in\widehat K_{\sigma}$, we can decompose $\sigma\restr L$ into
a direct sum $\bigoplus_{\tau\in\widehat L_{\omega}}
m_{\tau}^{\sigma}\cdot \tau$ for appropriate multiplicities
$m_{\tau}^{\sigma}$.  Thus $\H_{\sigma}=\bigoplus_{\tau\in\widehat
  L_{\omega}}\H_{m_{\tau}^{\sigma}}\tensor\H_{\tau}$, and
\eqref{eq:10} becomes
\begin{equation}
  \label{eq:11}
  \bigoplus_{\tau\in\widehat L_{\omega}} \Bigl(
  \bigoplus_{\sigma\in\widehat K_{\omega}} \ltgvs \tensor
  \H_{m_{\tau}^{\sigma}}\Bigr)
  \tensor\H_{\tau},
\end{equation}
and the restriction of $V\tensor \rho_{G}$ to $L$ is given by
\begin{equation*}
  (V\tensor\rho_{G})\restr L \cong \bigoplus_{\tau\in\widehat L_{\omega}}
  1_{\bigoplus_{\sigma\in \widehat
      K_{\omega}}\ltgvs\tensor\H_{m_{\tau}^{\sigma}}}\tensor\tau .
\end{equation*}
This induces an isomorphism of the fixed-point algebra
\begin{equation}
  \label{eq:12}
  \K(\H\tensor L^{2}(G))^{L,\Ad (V\tensor \rho_{G})} \cong
  \bigoplus_{\tau\in \widehat L_{\omega}} \K \Bigl(
  \bigoplus_{\sigma\in\widehat K_{\omega}} \ltgvs\tensor
  \H_{m_{\tau}^{\sigma}} \Bigr) \tensor 1_{\H_{\tau}}.
\end{equation}
Since the decomposition of $(V\tensor\rho_{G})\restr L$ into isotopes
$\K_{\tau}$ for the representations $\tau\in\widehat L_{\omega}$ is
unique, we conclude that
\begin{equation*}
  \K_{\tau} = \K(\ltgvt)\cong\K\Bigl( \bigoplus_{\sigma\in\widehat K_{\omega}}
  \ltgvs\tensor\H_{m_{\tau}^{\sigma}}\Bigr) ,
\end{equation*}
where as above, we have written $\ltgvt$ for
$L^{2}(G\times_{L,V\tensor \tau^{*}}(\H\tensor\H_{\tau}^{*}))$.

Summing up, we have the following proposition.
\begin{prop}
  \label{prop-mult-sig-in-tau}
  In the above setting, each block of type $\sigma\in\widehat
  K_{\omega}$ in the decomposition of $\K(\H\tensor
  L^{2}(G))^{K,\Ad(V\tensor\rho_{G})}$ appears with multiplicity
  $m_{\tau}^{\sigma}$ in each block of type $\tau\in\widehat
  L_{\omega}$ in the decomposition of $\K(\H\tensor
  L^{2}(G))^{L,\Ad(V\tensor \rho_{G})}$, where $m_{\tau}^{\sigma}$ is
  the multiplicity of $\tau$ in $\sigma\restr L$.
\end{prop}

We will also need to know how exterior equivalence effects our fixed
point algebras.  Recall that $\alpha,\beta:L\to \Aut A$ are called
exterior equivalent if there is a strictly continuous map $u:L\to
U(A)$ such that
\begin{equation}
  \label{eq:13}
  \alpha(l)={\Ad(u)(l)}\circ \beta(l)\quad\text{and}\quad
  u(lk)=u(l)\beta_{l}(u(k)) \quad\text{for all $k,l\in L$.}
\end{equation}
In our case, $A=\K(\H)$ and the map $u$ in \eqref{eq:13} is a strongly
continuous map into $U(\H)$ such that $\alpha=\Ad (uV)$.  Since $G$ is
second countable, we can choose a Borel cross-section $c:G/L\to G$ and
decompose $L^{2}(G)\cong L^{2}(G/L)\tensor L^{2}(L)$ by sending
$\xi\in L^{2}(G)$ to the element $\tilde \xi\in L^{2}(G/L)\tensor
L^{2}(L)\cong L^{2}(G/L\times L)$ given by $\tilde
\xi(sL,l)=\xi(c(sL)l)$.  Then we obtain a decomposition
\begin{equation*}
  \H\tensor L^{2}(G)\cong L^{2}(G/L)\tensor\H\tensor L^{2}(L)
\end{equation*}
which intertwines the representation $(V\tensor \rho_{G})\restr L$
with the representation $1_{L^{2}(G/L)}\tensor (V\tensor \rho_{L})$.
Then it follows from Lemma~\ref{lem-absorbtion} (applied twice), after
identifying $\H\tensor L^{2}(L)$ with $L^{2}(L,\H)$, that the unitary
$W:L^{2}(L,\H)\to L^{2}(K,\H)$ given by
\begin{equation}\label{eq:8}
  (W\xi)(l):= V(l)^{*}u(l)^{*}V(l)\xi(l)
\end{equation}
intertwines $V\tensor \rho_{L}$ and $uV\tensor \rho_{L}$.  Therefore,
we get the following result.
\begin{lemma}
  \label{lem-ext-equ}
  Let $W$ be as above and let $\tilde W\in U(\H\tensor L^{2}(G))$
  corresponding to $1_{L^{2}(G/L)}\tensor W$ under the decomposition
  $L^{2}(G)\cong L^{2}(G/L)\tensor L^{2}(L)$.  Then
  \begin{equation*}
    \K(\H\tensor L^{2}(G))^{L,\Ad(uV\tensor\rho_{G})} \cong \tilde W
    \bigl( \K(\H\tensor L^{2}(G))^{L,\Ad(V\tensor\rho_{G})}
    \bigr)\tilde W^{*}.
  \end{equation*}
\end{lemma}

\section{The Space $\stabxbh$}
\label{sec:space-stabxbh}

Our object in this section is to equip $\stabxbh$ with a natural
topology so that the induction map $\Ind:\stabxbh\to
\spec{B\rtimes_{\beta}G}$ of Theorem~\ref{thm-mackey-machine} induces
a homeomorphism of the quotient topological space $G\backslash
\stabxbh$ onto $\spec{B\rtimes_{\beta}G}$ for any {strictly} proper
action $\beta$ of $G$ on a continuous-trace \cs-algebra $B$ with
spectrum $X$.

{First, we need some general observations.  Let $A$ be a
  $\xtg$-algebra with respect to $\alpha:G\to \Aut A$ so that we can
  form the generalized fixed point algebra $A^{G,\alpha}$, and recall
  that $A^{G,\alpha}$ is a $C_{0}(\gux)$-algebra.  If $q:X\to\gux$ is
  the orbit map, we can form the pull-back
  \begin{equation*}
    q^{*}A^{G,\alpha}=C_{0}(X)\tensor_{C_{0}(\gux)}A^{G,\alpha}.
  \end{equation*}
  Using the description of the primitive ideal space from
  \cite{raewil:tams85}*{Lemma~1.1}, it is easy {to} identify the fibre
  $q^{*}A^{G,\alpha}(x)$ with $A^{G,\alpha}(G\cdot x)$.  Recall that
  Lemma~\ref{lem-genfixedalg} implies $A^{G,\alpha}(G\cdot x)$ is
  isomorphic to $A(x)^{G_{x},\alpha^{x}}$ via evaluation at $x$.

  We let
  \begin{equation*}
    A_{\fix}:=\set{a\in A=\sa_{0}(X,\A):\text{$a(x)\in
        A(x)^{G_{x},\alpha^{x}}$ for all $x$}}.
  \end{equation*}
  Then $A_{\fix}$ is a $C_{0}(X)$-subalgebra of $A$.  In fact it is
  $G$-invariant.  To see this, let $a\in \A_{\fix}$ and
  $b:=\alpha_{s}(a)$ for some $s\in G$.  Let $k\in G_{x}$.  Appealing
  to \eqref{eq:1} and \eqref{eq:2} as necessary, we have
  \begin{align*}
    \alpha_{k}^{x}\bigl(b(x)\bigr) &= \alpha_{x,k}\bigl(b(x)\bigr) = \alpha_{ks}(a)(x) \\
    &= \alpha_{s^{-1}\cdot x,ks} \bigl(a(s^{-1}\cdot x)\bigr) \\
    &= \alpha_{s^{-1}\cdot x,s}\circ \alpha_{s^{-1}\cdot
      x,s^{-1}ks}\bigl(a(s^{-1}\cdot x)\bigr) \\
    \intertext{which, since $s^{-1}ks\in G_{s^{-1}\cdot x}$,
      $\alpha_{s^{-1}\cdot x,s^{-1}ks}=\alpha_{s^{-1}ks}^{s^{-1}\cdot
        x}$ and $a\in A_{\fix}$, is}
    &=\alpha_{s^{-1}\cdot x,s}\bigl(a(s^{-1}\cdot x)\bigr) \\
    &=\alpha_{s}(a)(x)=b(x).
  \end{align*}

  \begin{lemma}
    \label{lem-closed-mult}
    Viewing $A^{G,\alpha}{\subseteq} \sa_{b}(X,\A)$ as in
    Lemma~\ref{lem-genfixedalg}, it follows that if $\phi\in C_{0}(X)$
    and $a\in A^{G,\alpha}$, then $\phi\cdot a\in A_{\fix}$ (where
    $(\phi\cdot a)(x)=\phi(x)a(x)$).  In particular, $A_{\fix}$ is a
    $C_{0}(X)$-algebra with fibres $A_{\fix}(x)\cong
    A(x)^{G_{x},\alpha^{x}}$ via evaluation at $x$.  Its spectrum can
    be identified with the set $\widehat
    A_{\fix}=\set{(x,\pi):\text{$x\in X$ and $\pi\in
        \spec{A(x)^{G_{x},\alpha^{x}}}$}}$.  Furthermore, the induced
    $G$-action on $\widehat A_{\fix}$ {is} given by $s\cdot
    (x,\pi)=(s\cdot x, \pi \circ \alpha_{s^{-1}\cdot x, s})$.
  \end{lemma}
  \begin{proof}
    The first assertion is straightforward as any $a\in A^{G,\alpha}$
    must satisfy $a(x)\in A(x)^{G_{x},\alpha^{x}}$. But evaluation at
    $x$ clearly defines an injection of $A_{\fix}(x)$ into
    $A(x)^{G_{x},\alpha^{x}}$.  But if $a_{0}\in
    A(x)^{G_{x},\alpha^{x}}$, then Lemma~\ref{lem-genfixedalg} implies
    that there is a $a\in A^{G,\alpha}$ such that $a(x)=a_{0}$.  We
    let $\phi\in C_{0}(X)$ be such that $\phi(x)=1$.  Then $\phi\cdot
    a\in A_{\fix}$ and $(\phi\cdot a)(x)=a_{0}$.  Hence $A_{\fix}$ is
    a $C_{0}(X)$-algebra with fibres as claimed.  The remaining
    assertions are straightforward consequences of the fact that the
    spectrum of a $C_{0}(X)$-algebra is the disjoint union of the
    spectrums of its fibres.
  \end{proof}

  \begin{prop}
    \label{prop-pull-back-iso}
    Let $A$ be a \xtg-algebra as above.  Then the map $\phi\tensor
    a\mapsto \phi\cdot a$ induces a $C_{0}(X)$-isomorphism of the
    pull-back $q^{*}A^{G,\alpha}$ onto $A_{\fix}$.
  \end{prop}
  \begin{proof}
    In view of Lemma~\ref{lem-closed-mult}, it is clear that
    $\phi\tensor a\mapsto \phi\cdot a$ induces a
    $C_{0}(X)$-homomorphism which is an isomorphism on the fibres.
    Hence it is an isomorphism as claimed.
  \end{proof}

  By Lemma~\ref{lem-genfixedalg}, $A^{G,\alpha}$ is a
  $C_{0}(\gux)$-algebra with evaluation at $x$ inducing an isomorphism
  of the fibres $A^{G,\alpha}(G\cdot x)$ with
  $A(x)^{G_{x},\alpha^{x}}$.  Hence the irreducible representations of
  $A^{G,\alpha}$ are given by pairs $[x,\pi]$ with $x\in X$ and
  $\pi\in \spec{A^{G_{x},\alpha^{x}}}$ so that
  $[x,\pi](a)=\pi\bigl(a(x)\bigr)$.  Furthermore $[x,\pi]=[y,\rho]$
  exactly when $(y,\rho)=(s\cdot x, \pi\circ \alpha_{s^{-1}\cdot x,
    s})$. Combining this with Lemma~\ref{lem-closed-mult}, we see that
  the map $(x,\pi)\mapsto [x,\pi]$ induces a bijection of $G\backslash
  \widehat A_{\fix}$ onto $\spec{A^{G,\alpha}}$.

  \begin{prop}
    \label{prop-gen-homeo}
    Let $A$ be a \xtg-algebra as above.  Then the map $(x,\pi)\mapsto
    [x,\pi]$ induces a homeomorphism of $G\backslash \widehat
    A_{\fix}$ and $\spec{A^{G,\alpha}}$.
  \end{prop}
  \begin{proof}
    There is a natural homomorphism of $A^{G,\alpha}$ into the
    multiplier algebra $M(A_{\fix})$ and $[x,\pi]$ is just the
    restriction of the natural extension of $(x,\pi)$ to
    $M(A_{\fix})$.  Hence $(x,\pi)\mapsto [x,\pi]$ is continuous by
    general nonsense (see \cite{gre:am78}*{Proposition~9}).

    Thus it will suffice to see that the map is open. For this, it
    suffices to show that if $[x_{i},\pi_{i}]\to [x,\pi]$, then we can
    pass to a subsequence, relabel, and find $(y_{i},\rho_{i})\to
    (x,\pi)$ such that $[y_{i},\rho_{i}]=[x_{i},\pi_{i}]$ (see
    \cite{wil:crossed}*{Proposition~1.15}).  However, since
    $A^{G,\alpha}$ is a $C_{0}(\gux)$-algebra, we must have $G\cdot
    x_{i}\to G\cdot x$, and after passing to a subsequence, relabeling
    and adjusting the $\pi_{i}$ as necessary, we can assume that
    $x_{i}\to x$.  But \cite{raewil:tams85}*{Lemma~1.1} implies that
    \begin{equation*}
      \spec{q^{*}A^{G,\alpha}}=\set{\bigl(x,[y,\rho]\bigr)\in X\times
        \spec{A^{G,\alpha}}:y\in G\cdot x}
    \end{equation*}
    has the relative product topology.  Hence
    $\bigl(x_{i},[x_{i},\pi_{i}]\bigr) \to \bigl(x,[x,\pi]\bigr)$ in
    $\spec{q^{*}A^{G,\alpha}}$.  But the isomorphism of
    $q^{*}B^{G,\alpha}$ with $A_{\fix}$ given in
    Proposition~\ref{prop-pull-back-iso} intertwines $(x,[x,\pi])$
    with $(x,\pi)$.  Hence we must have $(x_{i},\pi_{i})\to (x,\pi)$
    in $\widehat A_{\fix}$.
  \end{proof}

  We want to apply the previous discussion to $(A,\alpha)=(B\tensor
  \kltg, \beta\tensor\Ad\rho_{G})$. Proposition~\ref{prop-ind-gx}
  implies that the map $(x,\sigma)\mapsto (x,\ind_{x}\sigma)$ is a
  $G$-equivariant map of $\stabxbh$ onto
  $\bigl(\bigl(B\tensor\kltg\bigr)_{\fix}\bigr)^{\wedge}$.  This gives
  us a natural choice of a topology for $\stabxbh$.

  \begin{definition}
    \label{def-stab}
    We equip $\stabxbh$ with the topology making its identification
    with $\bigl(\bigl(B\tensor\kltg\bigr)_{\fix}\bigr)^{\wedge}$ a
    homeomorphism.
  \end{definition}

  Having made this definition, our desired description of the spectrum
  of $B\rtimes_{\beta}G$ follows immediately from
  Proposition~\ref{prop-gen-homeo} and Theorem~\ref{thm-EE2.14}.

  \begin{thm}
    \label{thm-main-spec}
    Suppose that $B$ is a {strictly} proper continuous-trace
    \cs-algebra with spectrum $X$.  Then the map
    \begin{equation*}
      \Ind:\stabxbh \to \spec{B\rtimes_{\beta}G}
    \end{equation*}
    is continuous and open, and hence factors through a homeomorphism
    of $G\backslash \stabxbh$ with $\spec{B\rtimes_{\beta}G}$.
  \end{thm}

}
\section{Actions of Countable Groups and Lie Groups}
\label{sec:acti-count-groups}

Our Theorem~\ref{thm-main-spec} reduces the problem of understanding
the topology on the spectrum of $B\rtimes_{\beta}G$ to understanding
the topology of $\stabxbh$.  In this section we want to give a nice
description of this topology first {in} the case that $G$ is discrete,
and then in the case $X=\widehat B$ is a manifold and $G$ is Lie group
acting smoothly on $X$.

Before stating our results, we need to recall some common terminology
regarding projective representations and the corresponding
$\omega$-representations.  First if $\tau$ is an
$\omega'$-representation and $\sigma$ is an $\omega$-representation of
the group $K$, then we say $\tau$ is a \emph{subrepresentation} of
$\sigma$ (written $\tau\le \sigma$) if $[\omega]=[\omega']$ in
$H^{2}(K,\T)$ and if $f:G\to\T$ is a Borel map such that
$\omega=\delta f\cdot \omega'$, then the $\omega$-representation
$\tau':= f\cdot \tau$ is unitarily equivalent to a subrepresentation
of $\sigma$.  {Note also that} if $L$ is a closed subgroup of $K$ and
if $\omega\in Z^{2}(K,\T)$, then the class of $\omega\restr L$ in
$H^{2}(L,\T)$ depends only on the class of $\omega$ in $H^{2}(K,\T)$.

Our main results are as follows.

\begin{thm}
  \label{thm-main1}
  Suppose that $\beta:G\to \Aut B$ is a {strictly} proper action of a
  countable group $G$ on a separable continuous-trace \cs-algebra $B$
  with spectrum $X$.  Then $(x_{n},\sigma_{n})\to (x,\sigma)$ in
  $\stabxbh$ if and only if
  \begin{enumerate}
  \item $x_{n}\to x$ in $X$, and
  \item there is a $N\in\N$ such that for all {$n\ge N$} we have
    $G_{x_{n}}{\subseteq} G_{x}$ and $\sigma_{n}\le
    \sigma\restr{G_{x_{n}}}$.
  \end{enumerate}
\end{thm}

We say that a Lie group action $\beta:G\to\Aut B$ on a \cs-algebra $B$
with spectrum $X$ is \emph{differentiable} if $X$ is a manifold and
each $s\in G$ acts by a diffeomorphism on $X$.
\begin{thm}
  \label{thm-main2}
  Suppose that $\beta:G\to \Aut B$ is a {strictly} proper
  differentiable action of a Lie group $G$ on a separable
  continuous-trace \cs-algebra $B$.  Then $(x_{n},\sigma_{n})\to
  (x,\sigma)$ if and only if
  \begin{enumerate}
  \item $x_{n}\to x$ in $X$, and
  \item each subsequence of $\sset{(x_{n},\sigma_{n})}$ has a
    subsequence $\sset{(y_{l},\rho_{l})}$ such that there is a
    sequence $\sset{s_{l}}$ in $G$ such that $s_{l}\to e$ in $G$ and a
    closed subgroup $L\subseteq G_{x}$ such that for all $l$,
    \begin{enumerate}
    \item $G_{s_{l}\cdot y_{l}}=s_{l}G_{y_{l}}s_{l}^{-1}=L$, and
    \item $s_{l}\cdot \rho_{l} \le \sigma\restr L$.
    \end{enumerate}
  \end{enumerate}
\end{thm}

First some preliminary observations.  Recall that if $D$ is a
\cs-subalgebra of $\K(\H)$, then $D\cong \bigoplus_{\tau\in\widehat D}
\K(\H_{\tau})$.  Let $p_{\tau}$ be the projection in $M(D)$
corresponding to $\K(\H_{\tau})$.  If $C$ is a \cs-subalgebra of $D$,
then $\tau\restr C$ maps $C$ onto
$C_{\tau}:=p_{\tau}Cp_{\tau}\subseteq \K(\H_{\tau})$.  If we decompose
$C_{\tau}\cong \bigoplus_{\sigma\in \widehat C_{\tau}}$, then
$\widehat C_{\tau}$ can be identified with those $\sigma\in \widehat
C$ which appear as subrepresentations of $\tau\restr C$; that is,
$\widehat C_{\tau}=\set{\sigma\in \widehat C:\sigma\le \tau\restr C}$.
We want to apply these observations to the following situation.

\begin{example}
  \label{ex-c-in-d}
  Let $L$ be a closed subgroup of a compact subgroup $K$ {of} $G$.
  Suppose that $V:K\to U(\H)$ is an $\omega$-representation.  Then
  \begin{equation*}
    D:=\bigl(\K\bigl(\H\tensor
    L^{2}(G)\bigr)\bigl)^{L,\Ad(V\tensor\rho_{G})}
  \end{equation*}
  is a subalgebra of $ \K(\H\tensor L^{2}(G))$ which contains the
  fixed-point algebra
  \begin{equation*}
    C:= \bigl(\K\bigl(\H\tensor L^{2}(G)\bigr)\bigr)^{K,\Ad(V\tensor\rho_{G})}.
  \end{equation*}
  Thus if we let $\ind_{K}:\widehat K_{\omega}\to \widehat C$ and
  $\ind_{L}:\widehat L_{\omega}\to \widehat D$ be the bijections from
  Remark~\ref{rem-summary}, then after combining the above
  considerations with Proposition~\ref{prop-mult-sig-in-tau}, we see
  that
  \begin{equation*}
    \ind_{K}\sigma\le (\ind_{L}\tau)\restr C\Longleftrightarrow \tau\le
    \sigma\restr L.
  \end{equation*}
\end{example}

\begin{lemma}
  \label{lem-n-infinity}
  Let $\N_{\infty}$ be the one-point compactification of $\N$ and let
  $C\subseteq D\subseteq \K(\H)$ be \cs-subalgebras.  Let
  \begin{equation*}
    A:=\set{f\in C(\N_{\infty},D):f(\infty)\in C}.
  \end{equation*}
  Then $\widehat A=(\N\times \widehat D)\coprod \widehat C$ and a
  sequence $\sset{(n,\rho_{n})}$ in $\N\times\widehat D$ converges to
  $\sigma\in \widehat C$ if and only if there is an $N\in\N$ such that
  $\sigma \le \rho_{n}\restr C$ for all $n\ge N$.
\end{lemma}
\begin{proof}
  Assume that $(n,\rho_{n})\to \sigma$ in $\widehat A$.  If the
  assertion in the lemma is false, then we can pass to a subsequence,
  relabel, and assume that for all $n$, $\sigma\not\le \rho_{n}\restr
  C$.

  Let $B=C(\N_{\infty}C)$ viewed as a subalgebra of $A$.  Since
  restriction gives a continuous map from $\Rep(A)\to \Rep(B)$ in the
  Fell topology (see \cite{ech:mams96}*{\S1.2}), we have
  $(n,\rho_{n})\restr C\to \sigma$ in $\Rep(B)$.  Then, identifying
  $C$ with the constant functions in $C(\N_{\infty},C)$, we see that
  $\rho_{n}\restr C\to \sigma$ in $\Rep(C)$.  Since $\rho_{n}\restr C$
  decomposes as a direct sum of irreducibles,
  \cite{sch:aim80}*{Theorem~2.2} implies that, after passing to a
  subsequence and relabeling, we can find irreducible
  subrepresentations $\sigma_{n}\le \rho_{n}\restr C$ such that
  $\sigma_{n}\to \sigma$ in $\widehat C$.  Since $\widehat C$ is
  clearly discrete, we eventually have $\sigma_{n}=\sigma$ which is a
  contradiction.

  Conversely, assume that $\sigma\le \rho_{n}\restr C$ for all $n\ge
  N$.  Let $a\in C(\N_{\infty},D)$ be such that
  $(n,\rho_{n})(a)=\rho_{n}(a(n))=0$ for all $n\in \N$.  We want to
  see that $a(\infty)=0$.  Let $\tilde a\in A$ be the constant
  function with value $a(\infty)$.  Then $\|a(n)-\tilde a(n)\|\to 0$
  with $n$.  Since $\sigma$ is a subrepresentation of $\rho_{n}\restr
  C$, this implies that
  \begin{equation*}
    \|\sigma(a(\infty))\|\le \|\rho_{n}(a(\infty))\| =
    \|\rho_{n}\bigl(a(n)-\tilde a(n)\bigr) \| \le \|a(n)-\tilde
    a(n)\|
  \end{equation*}
  for any $n\ge N$.  Hence $\sigma(a(\infty))=0$.  Since we can apply
  this to any subsequence of $\sset{(n,\rho_{n})}$, it follows that
  $(n,\rho_{n})\to \sigma$ as claimed.
\end{proof}

\begin{lemma}
  \label{lem-restr-to-ninf}
  Suppose that $\beta:G\to\Aut B$ is a {strictly} proper action of a
  second countable locally compact group on a separable, \emph{stable}
  continuous-trace \cs-algebra $B$ with spectrum $X$.  Assume that
  $\sset{x_{n}}$ is a sequence in $X$ converging to $x$ such that
  \begin{enumerate}
  \item for all $n\not=m$, $x_{n}\not= x_{m}\not=x$, and
  \item there is a fixed subgroup $L\subseteq G_{x}:=K$ such that
    $G_{x_{n}}=L$ for all $n$.
  \end{enumerate}
  Moreover, let $\pi:B\to \K(\H)$ be an irreducible representation of
  $B$ corresponding to $x$ and let $V:K\to U(\H)$ be an
  $\omega$-representation implementing the action of $K$ on $B(x)\cong
  \K(\H)$.  Let
  \begin{equation*}
    C:=\bigl(\K(\H\tensor
    L^{2}(G))\bigr)^{K,\Ad(V\tensor\rho_{G})}\quad\text{and}\quad D=\bigl(\K(\H\tensor
    L^{2}(G))\bigr)^{L,\Ad(V\tensor\rho_{G})}.
  \end{equation*}
  Then there exists $N\in \N$ such that, after identifying
  $S=\set{x_{n}:n\ge N}\cup \sset x$ with $\N_{\infty}$, we get an
  isomorphism~of
  \begin{equation*}
    \bigl(B\tensor\kltg\bigr)_{\fix,S}\quad\text{with}\quad \set{a\in
      C(\N_{\infty},D): a(\infty)\in C},
  \end{equation*}
  where $\bigl(B\tensor\kltg\bigr)_{\fix,S}$ is the restriction of the
  $C_{0}(X)$-algebra $(B\tensor\kltg)_{\fix}$ to the closed subset $S$
  of $X$.
\end{lemma}
\begin{proof}
  Let $B_{S}$ be the restriction of $B$ to $S\cong\N_{\infty}$.  Since
  $B$ is a stable continuous-trace \cs-algebra, $B_{S}\cong
  {C(\N_{\infty},\K(\H))}$ and the isomorphism transports the actions
  $\beta^{x_{n}}$ on $L=G_{x_{n}}$ on $B(x_{n})$ to appropriate
  actions $\beta^{n}$ of $L$ on $\K(\H)$.  Also let $\beta^{\infty}$
  be the restriction {of} the given $K$-action $\beta^{x}$ on $\K(\H)$
  to $L$.  We obtain a {$C(\N_{\infty})$}-linear action $\beta^{L}$ of
  $L$ on $C(\N_{\infty},\K(\H))$ by
  $\beta^{L}_{l}(a)(n)=\beta^{n}\bigl(a(n)\bigr)$.  Since every
  $\widehat L_{\ab}$-bundle over $\N_{\infty}$ is trivial, it follows
  from \cite{echwil:jot01}*{Theorem~5.4} that $\beta^{L}$ is
  classified up to exterior equivalence by the continuous
  Mackey-obstruction map $y\mapsto [\omega_{y}]\in H^{2}(L,\T)$ on
  $\N_{\infty}$.  Since $L$ is compact, $H^{2}(L,\T)$ is discrete
  {(combine \cite[Corollary 1]{moo:tams64} with the results at the
    beginning of \cite[Chapter III]{moo:tams64b} which imply that
    $H^2(L,\T)$ is a countable, locally compact Hausdorff group, hence
    discrete).}
  {Thus we} can assume that we have taken $N$ large enough so that
  $[\omega_{y}]=[\omega\restr L]$ for all $y$ (where $\omega$ is as in
  the statement of the lemma).  It follows from
  \cite{echwil:jot01}*{Theorem~5.4} that $\beta^{L}$ is exterior
  equivalent to the action defined by the constant field
  $\alpha=\id_{C(\N_{\infty})} \tensor \Ad V\restr L$.  Thus there is
  a continuous map $u:L\times\N_{\infty}\to U(\H)$ such that for all
  $n\in\N_{\infty}$, $\beta^{n}_{l}=\Ad\bigl(u(l, y)\cdot V_{l}\bigr)$
  and
  \begin{equation}
    \label{eq:6}
    u(lk,y)=u(l,y)V_{l}u(k,y)V_{l}^{*}\quad\text{for all $a\in
      \K(\H)$, $y\in \N_{\infty}$ and $l,k\in L$.}
  \end{equation}
  Since $\beta^{\infty}=\Ad V=\Ad\bigl(u(\cdot,\infty)V\bigr)$, it
  follows that $u(l,\infty)\in\T I_\H$ for all $l$.  Then \eqref{eq:6}
  implies that $l\mapsto u(l,\infty)$ is a character.  Multiplying
  each $u(\cdot,y)$ by the inverse of this character allows us to
  assume that $u(l,\infty)=1$ for all $l$.

  After identifying $S$ with $\N_{\infty}$, we have, by definition,
  \begin{multline*}
       \bigl(B\tensor\kltg\bigr)_{\fix,S} 
      \cong \bigl\{\, a\in C(\N_{\infty}, \K(\H\tensor L^{2}(G))) : \\
         \text{$a(y)\in \bigl(\K(\H\tensor
          L^{2}(G))\bigr)^{G_{y},\Ad(u(\cdot y)V\tensor
            \rho_{G})}$ for all $y\in \N_{\infty}$}\bigr\},
    \end{multline*} 
with $G_{y}=L$ if $y\in\N$ and $G_{\infty}=K$.

  For each $n\in\N$, let $W_{n}:L^{2}(G,\H)\to L^{2}(G,\H)$ be the
  unitary from Lemma~\ref{lem-ext-equ} corresponding to the cocycle
  $l\mapsto u(l,n)$ from $L$ into $U(\H)$.  It then follows -- see
  \eqref{eq:8} and recall that $u$ is strongly continuous with
  $u(l,\infty)=1$ for all $l$ -- that the sequence $\sset{W_n}$
  converges strongly to $1$.  Thus, if we define $W_{\infty}$ to be
  $1$, we get an isomorphism $\Phi:C(N_{\infty},\K(\H\tensor
  L^{2}(G)))\to C(\N_{\infty}, \K(\H\tensor L^{2}(G)))$ given by
  \begin{equation*}
    \Phi(a)(y)=W_{y}a(y)W_{y}\quad \text{for all $y\in \N_{\infty}$.}
  \end{equation*}
  Lemma~\ref{lem-ext-equ} implies that $\Phi$ maps $\bigl(B\tensor
  \kltg\bigr)_{\fix,S}$ onto 
\begin{multline*} \bigl\{a\in
        C(\N_{\infty},\K(\H\tensor L^{2}(G))):\\ \text{$a(y)\in
          \bigl(\K(\H\tensor L^{2}(G))\bigr)^{G_{y},\Ad(V\tensor
            \rho_{G})}$ for all $y\in \N_\infty$}\bigr\},
    \end{multline*}
  which in the notation of the lemma, is exactly $\set{a\in
    C(\N_{\infty},D):a(\infty)\in C}$.
\end{proof}

\begin{proof}[Proof of Theorem~\ref{thm-main1}]
  {If $\beta$ is a {strictly} proper action on $B$, then $\beta\tensor
    1$ is a {strictly} proper action on $B\tensor \K(\H)$.
    Furthermore, the corresponding ``$\fix$'' algebras, $(B\tensor
    L^{2}(G))_{\fix}$ and $(B\tensor\K(\H)\tensor L^{2}(G))_{\fix}$
    are Morita equivalent in such a way that the identification of the
    spectrum with $\stabxbh$ is preserved.  Thus we may as well assume
    from the onset that $B$ is stable (so that we can apply
    Lemma~\ref{lem-restr-to-ninf} when appropriate).}

  Now suppose that $(x_{n},\sigma_{n})\to (x,\sigma)$ in $\stabxbh$.
  Since $\bigl(B\tensor \kltg)\bigr)_{\fix}$ is a $C_{0}(X)$-algebra,
  we must have $x_{n}\to x$ in $X$.  Since $G$ is discrete, the action
  of $G$ on $X$ satisfies Palais's slice property (see
  Remark~\ref{rem-PS} and \cite{pal:aom61}*{Case 3 of
    Proposition~2.3.1}).  Hence we may assume that there is an open
  neighborhood $U$ of $x$ such that $G_{y}\subseteq G_{x}$ for all
  $y\in U$.  Thus there is {an} $N$ such that $n\ge N$ implies that
  $G_{x_{n}}\subseteq G_{x}$.  Hence it will suffice to assume that
  $x_{n}\to x$, $G_{x_{n}}\subseteq G_{x}$ for all $n$ and show that
  the failure of condition~(b) in the statement of the theorem results
  in a contradiction.

  Since $G_{x}$ must be a finite group, it can have only finitely many
  subgroups.  Thus we can pass to {a} subsequence and assume that
  there is a subgroup $L$ of $K:=G_{x}$ such that $G_{x_{n}}=L$ for
  all $n$ and such that condition~(b) fails for this sequence.  Let
  $S=\set{x_{n}:n\ge 1}\cup \sset x$.  As in the proof of
  Lemma~\ref{lem-restr-to-ninf}, we obtain a {$C(S)$}-linear action of
  $L$ on $C(S,\K(\H))$.  The Mackey obstruction at $x$ for this action
  is given by the restriction of the Mackey obstruction $[\omega_{x}]$
  of $K$ to $L$.  Since $x_{n}\mapsto [\omega_{x_{n}}]$ is continuous
  \cite{echwil:jot01}*{Lemma~5.3}, we have $[\omega_{x_{n}}]\to
  [\omega\restr L]$ in $H^{2}(L,\T)$.  But as $L$ is finite,
  $H^{2}(L,\T)$ is finite and there is {an} $N$ such that $n\ge N$
  implies that $[\omega_{x_{n}}]=[\omega_{x}\restr L]$.  We'll assume
  that this holds for all $n$.  Since we're assuming condition~(b)
  fails, there is no $N$ such that $\sigma_{n}\le \sigma\restr L$ for
  all $n\ge N$.  Passing to a subsequence, we can assume that
  $\sigma_{n}\not \le \sigma\restr L$ for all $n$.  The sequence
  $\sset{x_{n}}$ is either eventually constant {or} we can pass to a
  subsequence which, after relabeling, satisfies $x_{n}\not=
  x_{m}\not=x$ for all $n\not= m$.  In the eventually constant case,
  we have $\sigma_{n}\to \sigma$ in the discrete set $\widehat
  K_{\omega} \cong \bigl(\K(\H\tensor L^{2}(G))\bigr)^{K,\beta\tensor
    \Ad\rho}$.  But this implies that we eventually have
  $\sigma_{n}=\sigma$ which contradicts $\sigma_{n}\not\le \sigma$.
  Otherwise we are in the situation of Lemma~\ref{lem-restr-to-ninf}
  and we can assume that $(n,\sigma_{n})\to (\infty,\sigma)$ in
  $\set{a\in C(\N_{\infty},D):a(\infty)\in C}$ with $C$ and $D$ as in
  Lemma~\ref{lem-restr-to-ninf}.  Then a combination of
  Lemma~\ref{lem-n-infinity} and Example~\ref{ex-c-in-d} gives a
  contradiction.  Thus we have proved the forward implication in the
  Theorem.

  To prove the converse, let $(x_{n},\sigma_{n})$ and $(x,\sigma)$
  satisfy conditions (a)~and (b) in the Theorem.  It suffices to show
  that every subsequence of $\sset{(x_{n},\sigma_{n})}$ has a
  subsequence converging to $(x,\sigma)$.  Since every subsequence
  will still satisfy (a)~and (b), it will suffice to see that the
  given sequence has a subsequence converging to $(x,\sigma)$.  But
  either $\sset{x_{n}}$ must contain a constant subsequence equal to
  $x$ everywhere, or it has a subsequence satisfying the hypotheses of
  Lemma~\ref{lem-restr-to-ninf}.  In the case of a constant
  subsequence, condition~(b) translates to $\sigma_{n}=\sigma$ for all
  sufficiently large $n$ and this certainly implies convergence in
  $\stabxbh$.  In the second case, we can appeal to
  Lemma~\ref{lem-restr-to-ninf} which allows us to combine
  Lemma~\ref{lem-n-infinity} with Example~\ref{ex-c-in-d} to show that
  $(x_{n},\sigma_{n})\to (x,\sigma)$ in
  $\bigl(B\tensor\kltg\bigr)_{\fix,S} \cong \stabxbh$.
\end{proof}

We need a few more preliminaries before starting the proof of
Theorem~\ref{thm-main2}.
\begin{lemma}
  \label{lem-stab-iso}
  Suppose that $B$ is the section algebra of a continuous field of
  \cs-algebras over $X$, that $K$ is a compact subgroup of a second
  countable locally compact group $G$ and that $\beta:{K}\to\Aut B$ is
  an action of $K$ on $B$.  Then the algebras $A^{K}_{\fix}:=\set{b\in
    B\tensor\K(L^{2}(K)):b(x)\in \bigl(B(x)\tensor
    \K(L^{2}(K)\bigr)^{K_{x},\beta^{x}\tensor\Ad\rho_{K}}} $ and
  $A^{G}_{\fix}:=\set{b\in B\tensor \kltg:b(x)\in \bigl(B(x)\tensor
    \kltg\bigr)^{K_{x},\beta^{x}\tensor \Ad \rho_{G}} }$ are stably
  isomorphic.
\end{lemma}
\begin{proof}
  We will show that $A_{\fix}^{G}\cong A_{\fix}^{K}\tensor
  \K(L^{2}(G/K))$ when $G/K$ is equipped with the unique $G$-invariant
  measure such that $\int_{G}f(s)\,ds = \int_{G/K}\int_{K}
  f(sk)\,dk\,d\dot s$.  Since $G$ is second countable, there is a
  locally bounded Borel section for the quotient map $q:G\to G/K$.
  Then $c$ induces a Borel isomorphism of $G/K\times K $ onto $G$
  given by $(sK,k)\mapsto c(s)k$.  In turn, this induces an
  isomorphism $\Phi:L^{2}(G) \to L^{2}(G/K)\tensor L^{2}(K)$ given on
  continuous functions with compact support by
  $\Phi(\xi)(sK,k)=\xi(c(g)k)$.  It is clear that $\Phi$ transforms
  the restriction of the right-regular representation $\rho_{G}$ to
  $K$ to $1\tensor\rho_{K}$.  Thus we get an isomorphism of $B\tensor
  \kltg$ onto $B\tensor \K(L^{2}(K))\tensor \K(L^{2}(G/K))$ which
  transforms $\beta\tensor \Ad \rho_{G}\restr K$ with $\beta\tensor
  \Ad\rho_{K}\tensor\id$.  From this it is clear that
  $A^{G}_{\fix}\cong A^{K}_{\fix}\tensor \K(L^{2}(G/K))$.
\end{proof}

\begin{remark}
  The second countability assumption in Lemma~\ref{lem-stab-iso} is
  not necessary --- one always has a locally bounded Baire section and
  we could proceed as in \cite{wil:crossed}*{\S4.5} --- but the extra
  generality is unnecessary here so we omit the details.
\end{remark}

\begin{lemma}
  \label{lem-sub-o-rep}
  Suppose that $[\omega]\in H^{2}(K,\T)$ for some compact group $K$
  and let $\sigma:K\to U(\H_{\sigma})$ be an $\omega$-representation
  of $K$.  Suppose that $L$ is a closed subgroup of $K$ and $\rho:L\to
  U(\H_{\rho})$ is an $\omega\restr L$-representation of $L$ such that
  $\rho\le \sigma\restr L$.  For any $k\in K$, let $k\cdot
  \rho:kLk^{-1} \to U(\H_{\sigma})$ be given by $k\cdot
  \rho(klk^{-1})=\rho(l)$ for all $l\in L$.  Then $k\cdot \rho$ is a
  $\omega\restr{kLk^{-1}}$-representation with $k\cdot \rho\le
  \sigma\restr{kLk^{-1}}$.
\end{lemma}
\begin{proof}
  The result follows by applying the inner automorphism $C_{k}:K\to K$
  given by $C_{k}(l)=klk^{-1}$ to $\sigma$ and observing that $k\cdot
  \sigma=\sigma\circ C_{k}$ is equivalent to $\sigma$.
\end{proof}

An action of a group on a space $X$ is said to have \emph{locally
  finitely many orbit types} if for each $x\in X$ there is an open
$G$-invariant neighborhood $U_{x}$ of $x$ and a finite set of
conjugacy classes of subgroups of $G$ such that each stabilizer for
the action on $G$ on $U_{x}$ lies in one of these conjugacy classes.
A highly nontrivial, but classical result (see, for example,
\cite{bor:seminar60}*{Chap. VI~and VII}) implies that differentiable
actions of compact Lie groups on Manifolds have locally finitely many
orbit types.

\begin{lemma}
  \label{lem-fin-orb-types}
  Let $K$ be a second countable compact group acting on a separable
  stable continuous-trace \cs-algebra $B$ with spectrum $X$.  Suppose
  that $K$ fixes $x\in X$ and that the action has finitely many orbit
  types.  Let $\sigma\in \widehat K_{\omega_{x}}=\widehat
  K_{x,\omega_{x}}$.  Then $(x_{n},\sigma_{n})\to (x,\sigma)$ in
  \stabxbh\ if and only if
  \begin{enumerate}
  \item $x_{n}\to x$, and
  \item there is {an} $N\in\N$ such that $\sigma_{x_{n}}\le
    \sigma\restr{K_{x_{n}}}$ for all $n\ge N$.
  \end{enumerate}
\end{lemma}
\begin{proof}
  Suppose that $(x_{n},\sigma_{n})\to (x,\sigma)$.  Then we certainly
  have $x_{n}\to x$.  If condition~(b) fails, {we} can pass to a
  subsequence such that $\sigma_{x_{n}}\not\le
  \sigma\restr{K_{x_{n}}}$ for all $n$.  Since there are only finitely
  many orbit types, we can take this subsequence so that all the
  stability groups are conjugate.  Thus, we can also assume that there
  are $s_{n}\in K$ such that $K_{s_{n}\cdot x}$ is a constant subgroup
  $\tilde L$.  Since $K$ is compact, we can assume that $s_{n}\to s\in
  K$.  Let $h_{n}=s^{-1}s_{n}$ and let $L=s\tilde Ls^{-1}$.  Since the
  action of conjugation on $\stabxbh$ is continuous, we see that
  $(h_{n}\cdot x_{n},h_{n}\cdot \sigma_{n}):=(y_{n},\rho_{n}) \to
  (x,\sigma)$.  Passing to another subsequence, we assume that either
  $y_{n}=x$ for all $n$ or that $y_{n}\not=y_{m}\not=x$ for all
  $n\not=m$.  Just as in the proof of Theorem~\ref{thm-main1} we can
  appeal to Lemma~\ref{lem-restr-to-ninf} to conclude that we
  eventually have $[\omega_{y_{n}}]=[\omega_{x}\restr L]$ and
  $\rho_{n}\le \sigma \restr L$.  But then Lemma~\ref{lem-sub-o-rep}
  gives a contradiction.

  Conversely, if (a)~and (b) hold, then we can use similar arguments
  as in Theorem~\ref{thm-main1} to reduce to the situation of
  Lemma~\ref{lem-restr-to-ninf} and complete the proof.
\end{proof}

\begin{lemma}
  \label{lem-from-sig}
  Suppose that $\beta:G\to\Aut B$ is a dynamical system such that
  there is a $G$-equivariant map $\phi:\widehat B\to G/H$.  (Recall
  that the $G$-action on $\widehat B$ is given by $s\cdot\pi=\pi\circ
  \beta_{s}^{-1}$.)  Let $Z:=\phi^{-1}(\sset {eH})\subseteq \widehat
  B$.  Then $\widehat B$ is $G$-homeomorphic to
  $G\times_{H}Z:=H\backslash (G\times Z)$ via the map $\psi$ sending
  $[s,\pi]$ to $s\cdot \pi$.  (The action of $H$ on $G\times Z$ is
  given by $h\cdot (s,\pi)=(sh^{-1},h\cdot \pi)$.)
\end{lemma}
\begin{proof}
  {This follows easily from \cite{ech:pams90,ech:pams92} or
    \cite{wil:crossed}*{Proposition~3.53}. (See
    \cite{ech:xx11v2}*{Theorem~6.2 and Corollary~6.3} for more
    details.)}
\end{proof}

\begin{lemma}
  \label{lem-slice-top}
  Suppose that $K$ is a compact group acting on a topological space
  $Y$, and suppose that $K$ fixes $y\in Y$.  Let $\sset{y_{i}}$ be a
  net in $Y$ such that $K\cdot y_{i}\to K\cdot y$ in $K\backslash Y$.
  Then the net $\sset{y_{i}}$ converges to $y$ in $Y$.
\end{lemma}
\begin{proof}
  If the assertion is false, then after passing to a subnet and
  relabeling, {we} can assume that there is a neighborhood $U$ of $y$
  such that $y_{i}\notin U$ for all $i$.  But since the orbit map is
  open, we may as well assume that there are $k_{i}\in K$ such that
  $k_{i}\cdot y_{i}\to y$.  Sine $K$ is compact, we can even assume
  that $k_{i}\to k$.  But then $y_{i}=k_{i}^{-1}k_{i}\cdot y_{i}\to
  k^{-1}\cdot y=y$.  Hence $y_{i}$ is eventually in $U$ which is a
  contradiction.
\end{proof}

\begin{remark}[Palais's Slice Property]
  \label{rem-PS}
  As in \cite{echeme:em11}*{Definition~1.7}, we say that a group $G$
  acting properly on a locally compact space $X$ satisfied property
  (SP) (for Palais's slice property) if $X$ is \emph{locally induced
    from stabilizers} in that each point $x\in X$ has a neigborhood
  $U_{x}$ such that there is a closed $G_{x}$-invariant set
  $S_{x}\subseteq U_{x}$ such that $x\in S_{x}$ and such that the
  induced space $G\times_{G_{x}}S_{x}$ is $G$-homeomorphic to $U_{x}$
  via the map $[s,y]\mapsto s\cdot y$.  Note that
  Lemma~\ref{lem-slice-top} implies that finding $S_{x}$ is equivalent
  to finding a continuous $G$-map $\phi_{x}:U_{x}\to G/G_{x}$ with
  $S_{x}=\phi_{x}^{-1}(\sset{eG_{x}})$.  (In the case
  $U_{x}=G\times_{G_{x}}S_{x}$, we can define such a map by sending
  $[s,y]\mapsto sG_{x}$.)  We call $S_{x}$ a \emph{local slice at
    $x$}.

  Palais's Slice Theorem (\cite{pal:aom61}*{Theorem~2.3.3}) implies
  that every proper action of a Lie group on a locally compact space
  has property (SP).  Moreover, if $G$ acts differentiably on a
  manifold $X$, then $S_{x}$ can be taken to be a submanifold of
  $U_{x}$ such that the action of $G_{x}$ on $S_{x}$ is
  differentiable.  This follows from the construction of the slice in
  \cite{pal:aom61}*{\S2.2} (combine the first lemma of
  \cite{pal:aom61}*{\S2.2} with \cite{pal:aom61}*{Proposition~2.1.7}).
  Hence the action of $G_{x}$ on $S_{x}$, and therefore the action of
  $G$ on $X$, has locally finitely many orbit types.
\end{remark}

\begin{proof}[Proof of Theorem~\ref{thm-main2}]
  As in the proof of Theorem~\ref{thm-main1}, we can assume that
    $B$ is stable.  We let $\sset{(x_{n},\sigma_{n})}$ be a sequence
  in $\stabxbh$ and $(x,\sigma)\in\stabxbh$ such that $x_{n}\to x$.
  In view of Remark~\ref{rem-PS} above, we can also assume that
  $X\cong G\times_{G_{x}}Y$ for a local slice $Y$ at $x$ and that the
  corresponding action of $G_{x}$ on $Y$ has finitely many orbit
  types. Let $\phi=\phi_{x}:X\to G/G_{x}$ be the corresponding
  $G$-equivariant map.  Then we get a $G$-equivariant map
  $\psi:\stabxbh\to G/G_{x}$ by $\psi(z,\rho)=\phi(z)$.  Let
  $Z:=\psi^{-1}(\sset{eG_{x}})$.  Since $\stabxbh$ is
  $G$-equivariantly isomorphic to $(B\tensor\kltg)_{\fix}$, we obtain
  a $G$-homeomorphism
  \begin{equation*}
    \Phi:G\times_{G_{x}} Z\to \stabxbh
  \end{equation*}
  given by $\Phi\bigl([s,(y,\rho)]\bigr) = (s\cdot y,s\cdot \rho)$.
  Choose $s_{n}\in G$ and $(y_{n},\rho_{n})\in Z$ such that
  $(s_{n}^{-1}\cdot y_{n},s_{n}^{-1}\cdot \rho)=(x_{n},\sigma_{n})$
  for all $n$.  Since $Y$ is a slice at $x$, the map $G_{x}\cdot
  y\mapsto G\cdot y$ is a homeomorphism of $G_{x}\backslash Y$ onto
  $G\backslash X$.  Since $G\cdot x_{n}\to G\cdot x$ and $G\cdot
  y_{n}=G\cdot x_{n}$, we must have $G_{x}\cdot y_{n}\to G_{x}\cdot
  x$.  Thus $y_{n}\to x$ by Lemma~\ref{lem-slice-top}.

  Assume now that $(x_{n},\sigma_{n})\to (x,\sigma)$ in \stabxbh.
  Then we certainly have $x_{n}\to x$, so we can assume the set-up in
  the previous paragraph.  Then we claim that $(y_{n},\rho_{n})\to
  (x,\sigma)$ in $\stabxbh$.  Replacing $\sset{(y_{n},\rho_{n})}$ by a
  subsequence, it suffices to see that a subsequence converges to
  $(x,\sigma)$.  Since $y_{n}\to x$ and $x_{n}=s_{n}^{-1}\cdot
  y_{n}\to x$, the properness of the action allows us to pass to a
  subsequence and relabel and assume that $s_{n}\to s\in G_{x}$.  Then
  it follows from the continuity of the $G$-action on \stabxbh\ that
  $(y_{n},\rho_{n})= s_{n}^{-1}\cdot (x_{n},\sigma_{n})\to s^{-1}\cdot
  (x,\sigma)=(x,\sigma)$.  This proves the claim.

  We have $Z=\bigl(\bigl(B\tensor\kltg\bigr)_{\fix,Y}\bigr)^{\wedge}$.
  On the other hand, $G_{x}$ acts on $B\restr Y$ and
  Lemma~\ref{lem-stab-iso} implies that the corresponding space
  $\stabybh \cong\bigl( \bigl(B\tensor
  \K(L^{2}(G_{x})\bigr)_{\fix}\bigr)^{\wedge}$ is homeomorphic to $Z$.
  Thus it follows from Lemma~\ref{lem-fin-orb-types} that for large
  $n$ we have $[\omega_{s_{n}\cdot
    x_{n}}]=[\omega_{y_{n}}]=[\omega_{x}\restr{G_{x}}]
  =[\omega_{x}\restr{G_{s_{n}\cdot x_{n}}}]$ and $s_{n}\cdot
  \sigma_{n}=\rho_{n}\le \sigma\restr{G_{s_{n}\cdot x_{n}}}$.  Thus
  condition~(b) holds and we've established the forward implication of
  the theorem.

  The converse follows by applying Lemma~\ref{lem-n-infinity} and
  Lemma~\ref{lem-restr-to-ninf} as in the proof of
  Theorem~\ref{thm-main1} to the subsequence $\sset{(s_{l}\cdot
    y_{l},s_{l}\cdot \rho_{l})}$.  Since $G_{s_{l}\cdot
    y_{l}}=L\subseteq G_{x}$ and $s_{l}\cdot \rho_{l}\le \sigma\restr
  L$, it follows from those lemmas that $(s_{l}\cdot y_{l},s_{l}\cdot
  \rho_{l})\to (x,\sigma)$.  Thus condition~(b) tells us that every
  subsequence of $\sset{(x_{n},\sigma_{n})}$ has a subsequence
  converging to $(x,\sigma)$.  This implies $(x_{n},\sigma_{n})\to
  (x,\sigma)$ and completes the proof.
\end{proof}

\section{Application to group extensions}
\label{sec-app}

In this section we want to study of the unitary dual $\widehat{G}$ of
a locally compact group $G$ which fits into a short exact sequence
\begin{equation*}
  \xymatrix{1\ar[r]&N\ar[r]^{\iota}&G\ar[r]^{q}&K\ar[r]&1}
\end{equation*}
of locally compact groups in which $N$ is abelian and $K$ is compact.
Following Green (\cite{gre:am78} --- but see \cite[Chapter
1]{ech:mams96} for a survey) we may write the C*-group algebra
$C^*(G)$ as a twisted crossed product
$C^*(N)\rtimes_{\alpha,\tau}(G,N)\cong
C_0(\widehat{N})\rtimes_{\widehat\alpha,\widehat\tau}(G,N)$ in which
the action $\alpha$ is given by the conjugation action on the dense
subalgebra $C_c(N)\subseteq C^*(N)$ and the twist $\tau:N\to
UM(C^*(N))$ is given by the canonical inclusion map.  Then
$(\widehat{\alpha},\widehat{\tau})$ is the twisted action on
$C_0(\widehat{N})$ corresponding to $(\alpha, \tau)$ via the Fourier
isomorphism $C^*(N)\cong C_0(\widehat{N})$. Using a version of the
Packer-Raeburn stabilization trick (e.g., see \cite[Chapter
2]{ech:mams96}) we see that the twisted system is Morita equivalent to
an action $\beta$ of the compact group $K=G/N$ on
$B:=C_0(\widehat{N},\K)$ which covers the conjugation action of $K$ on
$\widehat{N}$.  Thus we are precisely in the situation of our general
results.

Note that Morita equivalence of twisted actions induces a Morita
equivalence of the twisted crossed products, hence a homeomorphisms
between the dual spaces of these crossed products. Moreover, it has
been worked out in \cite[Chapter 2]{ech:mams96} that passing to Morita
equivalent twisted actions is compatible with the Mackey-Rieffel-Green
machine of inducing representations from the stabilizers including the
computation of the Mackey obstructions (see \cite[Proposition 2.1.4 \&
2.1.5]{ech:mams96}).  Thus we see that $\widehat{G}$ is homeomorphic
to $G\backslash \stab(\widehat{N}_\beta)^\wedge$ with
$\stab(\widehat{N}_\beta)^\wedge$ topologized as in Definition
\ref{def-stab}.

In order to get a description of the space
$\stab(\widehat{N}_\beta)^\wedge$ we need to compute the
Mackey-obstructions of the twisted system $(C_0(\widehat{N}),
G,N,\widehat\alpha,\widehat\tau)$. Let $c:K\to G$ be a fixed choice of
a Borel section for the quotient map $q:G\to K$ such that $c(eN)=e$,
where $e$ denotes the unit of $G$. Then, for each $\chi\in
\widehat{N}$ we get a Borel extension $\tilde\chi:G\to\T$ of the
character $\chi$ by putting
$$\tilde\chi(s):=\chi\big(c(q(s))^{-1}s\big).$$
If we restrict this map to the stabilizer
\begin{equation*}
  G_\chi=\set{s\in
    G:\text{$\chi(sns^{-1})=\chi(n)$ for all $n\in N$}}
\end{equation*}
of the character $\chi$ in $G$, then we can check that
$(\epsilon_{\chi}, \tilde\chi)$, where
$\epsilon_{\chi}:C_0(\widehat{N})\to\C$ denotes evaluation at $\chi$,
is an $\om_{\chi}$-covariant representation of the twisted system
$(C_0(\widehat{N}), G_\chi, N,\widehat\alpha,\widehat\tau)$, with
multiplier $\om_\chi$ given by
\begin{equation}\label{eq-omchi}
  \om_\chi(s,t)=\tilde\chi(s)\tilde\chi(t)\tilde\chi(st)^{-1}.
\end{equation}
We then compute
\begin{align*}
  \om_\chi(s,t)&=\chi\big(c(q(s))^{-1}s\big)
  \chi\big(c(q(t))^{-1}t\big)\chi\big(t^{-1}s^{-1} c(q(st))\big)\\
  &=\chi\big(c(q(s))^{-1}s\big)\chi\big(c(q(t))^{-1}
  s^{-1}c(q(st))\big)\\
  \intertext{which, since $\chi$ is invariant under conjugation with
    elements in $G_{\chi}$, is} &=\chi\big(c(q(s))^{-1}s\big)
  \chi\big(s^{-1}c(q(st))c(q(t))^{-1} \big)\\
  &=\chi\big(c(q(s))^{-1}c(q(st))c(q(t))^{-1} \big).
\end{align*}
Thus we see that $\om_{\chi}$ factors through a cocycle on the
stabilizer $K_\chi:=G_\chi/N$ of $\chi$ in $K$ and we obtain
\begin{equation*}
  \stab(\widehat{N}_\beta)^\wedge=\set{(\chi,\sigma): \text{$\chi\in
      \widehat{N}$ and $ \sigma\in \widehat{K}_{\chi,[\om_{\chi}]}$}}
\end{equation*}
with $\om_\chi$ as in the above computations. Now if $K$ is a compact
Lie group, we may apply Theorem~\ref{thm-main2} (or
Theorem~\ref{thm-main1} if $K$ is finite) to obtain a description of
the topology on $\stab(\widehat{N}_\beta)^\wedge$ in terms of
convergent sequences.

\begin{example} [The group $G=\mathbf{p4g}$] We want to illustrate the
  above procedure in the particular example of the crystallographic
  group $G=\mathbf{p4g}$, which is the subgroup of the full motion
  group $\R^2\rtimes\Otwo$ generated by $\{(n,E): n\in \Z^2\}$
  together with the elements
$$\{(0,R), (0,R^2), (0,R^3), (v, S), (v, SR), (v, SR^2), (v, SR^3)\}$$
where $R=\left(\begin{smallmatrix} 0&-1\\1&0\end{smallmatrix}\right)$,
$S=\left(\begin{smallmatrix} 1&0\\0&-1\end{smallmatrix}\right)$ and
$v=\left(\begin{smallmatrix}1/2\\ 1/2\end{smallmatrix}\right)\in
\R^2$. Then $G$ fits into an extension
\begin{equation*}
  \xymatrix{0\ar[r]&\Z^{2}\ar[r]&G\ar[r]&D_{4}\ar[r]&0}
\end{equation*}
where $D_4=\{E,R,R^2,R^3,S, SR, SR^2, SR^3\}$ is the dihedral
group. The quotient map is given by projection on the second factor
and we have an obvious section $c:D_4\to G$ given by $c(X)=(0,X)$ if
$X\in \lk R\rk$ and $c(X)=(v,X)$ if $X\in S\lk R\rk$. The conjugation
action of $D_4=G/\Z^2$ on $\T^2=\widehat{\Z^2}$ is given by matrix
multiplication; that is, if we write
$$\exp_2:\R^2\to\T^2;\; \exp_2\left(\begin{smallmatrix}s
    \\t\end{smallmatrix}\right)=\bigl(\begin{smallmatrix} e^{2\pi i
    s}\\
  e^{2\pi it}\end{smallmatrix}\bigr),$$ then we have
$$X\cdot \exp_2\left(\begin{smallmatrix}s\\t\end{smallmatrix}\right)
=\exp_2\big(X\left(\begin{smallmatrix}s\\t\end{smallmatrix}\right)\big)$$ 
for $X\in D_4$, $s,t\in \R$. This action has been studied in
\cite{echeme:em11} (see Examples 2.6, 3.5 and 4.10 of that paper). In
particular, it is shown in \cite{echeme:em11} that the image under the
exponential map of the triangle
$$Z:=\set{\left(\begin{smallmatrix}s\\t\end{smallmatrix}\right)\in \R^2:
\text{$0\leq t\leq \textstyle{\frac{1}{2}}$ and $ 0\leq s\leq t$}}$$ 
is a topological fundamental domain for the action of $D_4$ on $\T^2$
in the sense that the mapping $Z\to D_4\backslash \T^2; z\mapsto D_4z$
is a homeomorphism. It follows from this that
$$\stab(Z)^\wedge=\set{(\chi_z,\sigma): z\in Z, \sigma\in
  \widehat{D}_{z,[\om_z]}}$$ is a topological fundamental domain for
$\stab(\T^2_\beta)^\wedge$ for the action $\beta$ of $D_4$ on
$C(\T^2,\K)$ as described in the discussion above, where we write
$\chi_z:\Z^2\to \T$ for the character $\chi_z(n)=e^{2\pi i \lk
  z,n\rk}$ for $z\in \R^2$, $D_z$ for the stabilizer of $\chi_z$ in
$D_4$ and $\om_z$ for the Mackey-obstruction at $\chi_z$ as in
\eqref{eq-omchi}. It therefore follows from
Theorem~\ref{thm-main-spec} and the above considerations that
$\widehat{G}$ is homeomorphic to $\stab(Z)^\wedge$ (which we regard as
a closed subset of $\stab(\T^2_\beta)^\wedge$).

We want to study the space $\stab(Z)^\wedge$ more closely. For this we
first note that the action of $D_4$ on $Z$ (or rather its image in
$\T^2$) has the following stabilizers:
\begin{itemize}
\item $D_{\left(\begin{smallmatrix}s\\t\end{smallmatrix}\right)}=\{E\}
  $ if $0< s< t< \frac{1}{2}$;
\item $D_{\left(\begin{smallmatrix}s\\s\end{smallmatrix}\right)}=\lk
  SR^3\rk=:K_1$ if $0<s<\frac{1}{2}$;
\item $D_{\left(\begin{smallmatrix}0\\t\end{smallmatrix}\right)}=\lk
  SR^2\rk=:K_2$ if $0<t<\frac{1}{2}$;
\item $D_{\left(\begin{smallmatrix}s\\1/2\end{smallmatrix}\right)}=\lk
  S\rk=:K_3$ if $0< s <\frac{1}{2}$;
\item $D_{\left(\begin{smallmatrix}0\\1/2\end{smallmatrix}\right)}=\lk
  S, R^2\rk=:H$, and
\item
  $D_{\left(\begin{smallmatrix}0\\0\end{smallmatrix}\right)}=
  D_{\left(\begin{smallmatrix}1/2\\1/2\end{smallmatrix}\right)} 
  =D_4$.
\end{itemize}
We need to compute the cocycles $\om_z:=\om_{\chi_z}$ on the
stabilizers $D_z$ and the $\om_z$-representations of $D_z$ for all
$z\in Z$. Clearly, for $z$ in the interior $Z^\circ$ of $Z$ we have
trivial stabilizers and the trivial representations on these
stabilizers. As a consequence, the portion of $\stab(Z)^\wedge$
corresponding to $Z^\circ$ is homeomorphic to $Z^\circ$.  \medskip

To study the boundary points, we first observe that the cocycle
$\partial c\in Z^2(D_4, \Z^2)$ for the cross-section $c:D_4\to G$ can
be computed as
$$\partial c(X,Y)=c(X)^{-1}c(XY)c(Y)^{-1}
=\begin{cases} 0 & \text{if $Y\in \lk R\rk$}\\
    X^{-1}v-v & \text{if $Y\in S\lk R\rk$,  $X\in \lk R\rk$}\\
    -(X^{-1}v+v) &\text{if $Y,X \in S\lk R\rk$.}\end{cases}$$
If we plug this into the characters $\chi_z$ we get the following
cocycles on the stabilizers:

\medskip
\noindent {\textsc{Case 1.}} On the interiors of the edges of $Z$ with
stabilizers $K_1, K_2, K_3$ we get the values
\begin{itemize}
\item
  $\om_{\left(\begin{smallmatrix}s\\s\end{smallmatrix}\right)}(SR^3,
  SR^3)=e^{-4\pi i s}$ on $K_1=\lk SR^3\rk$, $0<s<\frac{1}{2}$;
\item
  $\om_{\left(\begin{smallmatrix}0\\t\end{smallmatrix}\right)}(SR^2,
  SR^2)=e^{-2\pi i t}$ on $K_2=\lk SR^2\rk$, $0<t<\frac{1}{2}$;
\item
  $\om_{\left(\begin{smallmatrix}s\\ \frac12\end{smallmatrix}\right)}(S,
  S)=e^{2\pi i s}$ on $K_3=\lk S\rk$, $0<s<\frac{1}{2}$;
\end{itemize}
(and all other values $1$).  Note that these cocycles, regarded as
$\T$-valued cocycles, are cohomologous to the trivial one.  In fact if
$u$ is any element of $\T$, if $\epsilon$ is the generator of
$\Z/2\Z$, and if $\om^u$ is the the $\T$-valued cocycle given by
$\om^u(\epsilon,\epsilon)=u$ (and all other values $1$), then $\om^u
=\partial \mu_w$ if $w\in \T$ is a square root of $u$ and
$\mu_w:\Z/2\Z\to \T$ is given by $\mu_w(1)=1$ and $\mu_w(\epsilon)=w$.
It follows in particular that if $w_1, w_2$ are the two roots of $u$,
then $\mu_{w_1}, \mu_{w_2}$ are the two irreducible
$\om^u$-representations of $\Z/2\Z$. This general description can be
applied to the cocycles on the groups $K_1,K_2, K_3$ considered above.

It follows then from Theorem~\ref{thm-main1} that on each open line
segment of $\partial Z$ the corresponding portion of $\stab(Z)^\wedge$
is homeomorphic to $(0,1)\times\{1,-1\}$. For example, on the segment
$S_1:=\{ \left(\begin{smallmatrix} s\\ s\end{smallmatrix}\right): 0<
s< \frac{1}{2}\}$ the homeomorphism $(0,1)\times\{1,-1\}\to
\stab(S_1)^\wedge$ is given by $(t, \pm1)\mapsto \big((\frac{t}{2},
\frac{t}{2}), \mu_{t,\pm 1}\big)$ with $\mu_{t, \pm1}(SR^3)=\pm e^{
  \pi i t}$ and similarly on $S_2:=\{\left(\begin{smallmatrix} 0\\
    t\end{smallmatrix}\right): 0<t<\frac{1}{2}\}$ and
$S_3:=\{\left(\begin{smallmatrix} s\\
    \frac{1}{2}\end{smallmatrix}\right): 0< s<\frac{1}{2}\}$.

\medskip
\noindent \textsc{Case 2.} For
$z_0=\left(\begin{smallmatrix}0\\0\end{smallmatrix}\right)$ we get the
trivial cocycle $\om_{z_0}\equiv 1$ on $D_4$ and the ordinary unitary
dual of $D_4$ which consists of the representations $\{\mu_0, \mu_1,
\mu_2,\mu_3, \lambda\}$ in which $\{\mu_0, \mu_1,\mu_2,\mu_3\}$ denote
the characters of the commutative quotient $D_4/\lk R^2\rk\cong \lk
\tilde{R}\rk\times\lk \tilde{S}\rk$ (where $\tilde{R}$ and $\tilde{S}$
denote the images of $R$ and $S$ in $D_4/\lk R^2\rk$) and $\lambda:
D_4\to M_2(\C)$ denotes the representation given by the inclusion
$D_4\subseteq U(2)$.

\medskip
\noindent \textsc{Case 3.}  For
$z_1:=\left(\begin{smallmatrix}0\\1/2\end{smallmatrix}\right)$ and
$D_{z_1}=H=\lk R^2,S\rk$ we get the values:
 $$\om_{z_1}(R^2, Y)
 =\om_{z_1}(SR^2, Y)=-1$$ for $Y\in \{S, SR^2\}$ and $1$ for all other
 values.  Since $\om_{z_1}$ takes it values in $C_2:=\{1,-1\}$, we may
 use the central extension $H\times_{\om_{z_1}}\!\!C_2$ which, as a
 set, is the direct product $H\times C_2$ with multiplication given by
$$(X, u)(Y, v)=(XY, \om_{z_1}(X,Y) uv)$$
for the study of the $\om_{z_1}$-representations of $H$. In fact,
there is a one-to-one correspondence between the irreducible
$\om_{z_1}$-representations $\sigma$ of $H$ and the irreducible
unitary representations $\tilde\sigma$ of $H\times_{\om_{z_1}}\!\!C_2$
which satisfy $\tilde\sigma(E, -1)=-1$ given as follows: if $\sigma\in
\widehat{H}_{[\om_{z_1}]}$ is given, the corresponding representation
$\tilde\sigma$ is given by $\tilde\sigma(X, u)= u\sigma(X)$ and if
$\tilde\sigma$ is given, then $\sigma(X)=\tilde\sigma(X, 1)$ is the
corresponding projective representation of $H$.

Now a short computation shows that $D_4$ is isomorphic to
$H\times_{\om_{z_1}}\!\!C_2$ by sending $R$ to $(SR^2,1)$ and $S$ to
$(R^2,1)$. If we then identify $\widehat{D}_4$ with
$(H\times_{\om_{z_1}}\!\!C_2)^\wedge$ via this isomorphism, we see
that only the two-dimensional representation $\lambda$ of $D_4$
corresponds to a representation of $H\times_{\om_{z_1}}\!\!C_2$ which
takes value $-1$ on $(E,-1)$. So we only get one (class) $[\sigma]$ of
irreducible $\om_{z_1}$-representations of $H$ given by $\sigma: H\to
U(2)$,
$$\sigma(E)=1,\quad\sigma(R^2)=\left(\begin{smallmatrix} 1&0\\0&-1\end{smallmatrix}\right),\quad
\sigma(S)=\left(\begin{smallmatrix}
    0&1\\1&0\end{smallmatrix}\right)\quad\text{and} \quad
\sigma(SR^2)=\left(\begin{smallmatrix}
    0&-1\\1&0\end{smallmatrix}\right).$$

\medskip
\noindent \textsc{Case 4.}  Finally, for
$z_2=\left(\begin{smallmatrix}1/2\\1/2\end{smallmatrix}\right)$ we get
the cocycle $\om_{z_2}$ on $D_4$ with values
$$\om_{z_2}(R, Y)=\om_{z_2}(R^3,Y)=\om_{z_2}(S,Y)=\om_{z_2}(SR^2,Y)=-1$$
for all $Y\in S\lk R\rk$ and $1$ otherwise. To study the
$\om_{z_2}$-representations of $D_4$ we proceed as above by studying
the central extension $L:=D_4\times_{\om_{z_2}}\!\!C_2$ and determine
those representations $\tau\in \widehat{L}$ which satisfy
$\tau(E,-1)=-1$.  Let's denote this set by $\widehat{L}^-$.  One
checks that $L$ is generated by the elements $(R,1)$ and $(S,1)$ and
with a little work we see that $\widehat{L}^-$ contains
\begin{itemize}
\item four one-dimensional representations $\zeta_0,\ldots, \zeta_3$
  given on these generators by
  \begin{align*}
    \zeta_0(R,1)&= \zeta_0(S,1)=i,\\
    \zeta_1(R,1)&=-\zeta_1(S,1)=i,\\
    \zeta_2(R,1)&=-\zeta_2(S,1)=-i,\\
    \zeta_3(R,1)&=\zeta_3(S,1)=-i, \quad\text{and}
  \end{align*}
\item one two-dimensional representation $\tau: L\to U(2)$ given on
  the generators by
$$\tau(R,1)=S=\left(\begin{smallmatrix} 1&\phantom{-}0\\0&-1\end{smallmatrix}\right)\quad{and}\quad 
\tau(S,1)=R= \left(\begin{smallmatrix}
    0&-1\\1&\phantom{-}0\end{smallmatrix}\right)$$
\end{itemize}
Restricting these representations to $D_4\times \{1\}\subseteq L$
gives the desired $\om_{z_2}$-representations of $D_4$ (which we shall
call by the same letters below).

\medskip By the above computations together with
Theorem~\ref{thm-main1} we get the following description of
$\stab(Z)^\wedge\cong \widehat{G}$ together with its topology: As a
\emph{set} we have the disjoint union
\begin{align*}
  \stab(Z)^\wedge= &Z^\circ \djunion \; (S_1\djunion  S_2\djunion  S_3)\times \{1,-1\}\\
  &\djunion  \;\{z_0\}\times \{\mu_0,\ldots, \mu_4,\lambda\}\\
  &\djunion \;\{z_1\}\times  \{\sigma\}\\
  &\djunion \;\{z_2\}\times \{\zeta_0,\ldots, \zeta_3,\tau\}.
\end{align*}
with topology restricted to the single ingredients the usual
one. Moreover, $Z^\circ$ is open in $\stab(Z)^\wedge$ and if we
approach any boundary point $\bar{z}\in \partial Z$ by a sequence
$(z_n)_n\in Z^\circ$, then this sequence converges to every element in
$\stab(Z)^\wedge$ which corresponds to $\bar z$.

Moreover, $(S_1\djunion S_2\djunion S_3) \times \{1,-1\}$ is open in
$\stab(\partial Z)^\wedge$.  Assume that $(z_n, u_n)_n$ is a sequence
in this set such that $z_n$ converges to one of the vertices $z_0,
z_1, z_2$. By passing to a subsequence, if necessary, we may assume
that $u_n=\pm1$ is constant. Using Theorem \ref{thm-main1} it is then
very easy to describe the limit points of this sequence in
$\stab(Z)^\wedge$.

We will work this out for the case where
$z_n=\left(\begin{smallmatrix} s_n\\
    \frac{1}{2}\end{smallmatrix}\right)\in S_3$ and we leave the other
cases to the reader.  The identification $S_3\times \{1,-1\}\cong
\stab(S_3)^\wedge$ is given by the map $\big(\left(\begin{smallmatrix}
    s\\ \frac{1}{2}\end{smallmatrix}\right), \pm1\big)\mapsto
\big(\left(\begin{smallmatrix} s\\
    \frac{1}{2}\end{smallmatrix}\right), \mu_{s, \pm 1}\big)$ with
$\mu_{s,\pm 1}(S)=\pm e^{\pi i s}$.
Since $\mu_{\frac{1}{2}, \pm 1}(S)=e^{\pi i\frac{1}{2}}=\pm i$ we
easily that $\mu_{\frac{1}{2}, 1}=\zeta_0|_{K_3}=\zeta_2|_{K_3}$ and
$\mu_{\frac{1}{2}, - 1}=\zeta_1|_{K_3}=\zeta_{3}|_{K_3}$.  Moreover,
$\tau(S)=R$ has the eigenvalues $\pm 1$, so that both,
$\mu_{\frac{1}{2}, 1}$ and $\mu_{\frac{1}{2}, -1}$ are
subrepresentations of $\tau|_{K_3}$. Thus, we see that the sequence
$\big(\left(\begin{smallmatrix} s_n\\
    \frac{1}{2}\end{smallmatrix}\right), \mu_{s_n, 1}\big)_n$
converges to $(z_2, \zeta_0), (z_2, \zeta_2), (z_2,\tau)$ and
$\big(\left(\begin{smallmatrix} s_n\\
    \frac{1}{2}\end{smallmatrix}\right), \mu_{s_n, 1}\big)_n$
converges to $(z_2, \zeta_1), (z_2,\zeta_3)$ and $(z_2, \tau)$.

Similarly, if $s_n\to 0$, then $\mu_{0, \pm1}(S)=\pm 1$. If $\sigma$
is the unique irreducible $\om_{z_1}$-representation, we have
$\sigma(S)=\left(\begin{smallmatrix} 0&1\\1&0\end{smallmatrix}\right)$
which has eigenvalues $\pm1$.  thus we see that $\mu_{0, 1}$ and
$\mu_{0, -1}$ are both sub representations of $\sigma|_{K_2}$, hence
both sequences $\big(\left(\begin{smallmatrix} s_n\\
    \frac{1}{2} \end{smallmatrix}\right), \mu_{s_n, \pm 1}\big)_n$
converge to $(z_1, \sigma)$.
\end{example}



\def\noopsort#1{}\def\cprime{$'$} \def\sp{^}

\end{document}